\documentclass[11pt]{amsart}

\usepackage{amssymb,amsmath,amsfonts,amsthm}
\usepackage[all]{xy}
\usepackage{marginnote}

\newtheorem{thm}{Theorem}[section]
\newtheorem{lem}[thm]{Lemma}
\newtheorem{cor}[thm]{Corollary}

\theoremstyle{remark}
\newtheorem{rem}[thm]{Remark}

\theoremstyle{definition}
\newtheorem{definition}[thm]{Definition}
\newtheorem{ex}[thm]{Example}

\title{Extension of Positive Definite functions}
\author{Palle E.T. Jorgensen and Robert Niedzialomski}
\address{Department of Mathematics, The University of Iowa, 14 MacLean Hall, Iowa City, IA-52242, USA} 
\email{palle-jorgensen@uiowa.edu}
\address{Department of Mathematics and Statistics, The University of Tennessee at Martin, 424 Humanities Building, Martin, TN-38238, USA}
\email{rniedzia@utm.edu}

\subjclass[2010]{47B32, 42A82, 47B25, 22D10}
\keywords{Positive definite functions, reproducing-kernel Hilbert spaces, symmetric(Hermitian) operators (unbounded), commutativity,  commuting selfadjoint extensions, spectral resolution, unitary representations}
\date{}

\begin{document}

\begin{abstract} Let $\Omega\subset\mathbb{R}^n$ be an open, connected subset of $\mathbb{R}^n$, and let $F\colon\Omega-\Omega\to\mathbb{C}$, where $\Omega-\Omega=\{x-y\colon x,y\in\Omega\}$, be a continuous positive definite function. We give necessary and sufficient conditions for $F$ to have an extension to a continuous positive definite function defined on the entire Euclidean space $\mathbb{R}^n$. The conditions are formulated in terms of strong commutativity of a system of certain unbounded selfadjoint operators defined on a Hilbert space associated to $F$. When a positive definite function $F$ is extendable, we show that it is characterized by existence of associated unitary representations of $\mathbb{R}^n$. Different positive definite extensions correspond to different unitary representations. We prove that each such unitary representation has simple spectrum. We give necessary and sufficient conditions for a continuous positive definite function to have exactly one extension. Our proof regarding extensions of positive definite functions carries over mutatis mutandis to the case of conditionally negative definite functions, which has applications to Gaussian stochastic processes, whose increments in mean-square are stationary (e.g., fractional Brownian motion).
\end{abstract}

\maketitle

\section{Introduction}

We consider the extension problem for continuous positive definite functions defined on an open subset of the Euclidean space $\mathbb{R}^n$. By ``extension" we always mean an extension to a continuous and positive definite function defined on the entire space $\mathbb{R}^n$. We give necessary and sufficient conditions for such a function to have an extension. The conditions (see Theorem \ref{MTh}) are formulated in terms of strong commutativity of a certain system of selfadjoint unbounded operators defined on a Hilbert space associated to our positive definite function. 

Let $\Omega\subset\mathbb{R}^n$ be an open set and let $F\colon\Omega-\Omega\to\mathbb{C}$ be a function, where $\Omega-\Omega=\{x-y\colon x,y\in\Omega\}$. We say that $F$ is \textit{positive definite} if for any $x_1,\ldots,x_m\in\Omega$ and any $c_1,\ldots,c_m\in \mathbb{C}$ we have
\begin{equation}\label{pdfun}
\sum_{j,k=1}^m F(x_j-x_k)c_j\overline{c_k}\geq 0.
\end{equation}
In particular, if $F$ is a positive definite function, then $F(0)\geq 0$ and for any $z\in \Omega-\Omega$
\begin{equation*}
F(-z)=\overline{F(z)}\quad\textrm{and}\quad |F(z)|\leq F(0).
\end{equation*}

\begin{rem} The assumption that $\Omega$ is an open set is not necessary for the definition, and the definition works for any set $\Omega\subset\mathbb{R}^n$, or even, for example, any subset of an abelian group.
\end{rem}

\begin{rem} If $\Omega\subset \mathbb{R}^n$ is a symmetric ($x\in\Omega\Rightarrow -x\in\Omega$) and convex set, then $\Omega-\Omega=2\Omega=\{2x\colon x\in\Omega\}$. 

\end{rem}

Below, we relate two extension problems, which on the surface might seem unrelated. One is the question of extending positive definite functions defined initially only on a fixed open subset $\Omega$ of $\mathbb{R}^n$; and the other is the question 
of finding $n$ commuting selfadjoint operators, which extend $n$ \textit{Hermitian operators}, which formally commute on a dense domain in a Hilbert space. In the case when a given locally defined positive definite function is extendable we show that it is characterized by existence of a certain system of commuting selfadjoint extension operators; or, equivalently, existence of associated unitary representations $U$ of $\mathbb{R}^n$; ``unitary representation'' in this paper means a strongly continuous one. Different positive definite extensions to $\mathbb{R}^n$ of the initial function correspond to different unitary representations.  Moreover (see corollary \ref{intert}) we find the complete spectral transform for each of the admissible $U$; and in particular, we prove that each such unitary representation $U$ has simple spectrum.

In earlier investigations, the question of existence of commuting selfadjoint extension operators has played a role in the study of duality questions in harmonic analysis. For example, in \cite{Fug74}, Fuglede showed 
that the following two questions for domains in $\mathbb{R}^n$ are directly related to existence of commuting selfadjoint extension operators; and subsequently (see e.g., \cite{IKT03}) these two questions have received a great deal of attention. Both are
existence problems. One question is geometric: tiling by translations, and the other asks whether $L^2(\Omega)$ has an orthogonal basis of Fourier frequencies.

We stress that different extension problems considered in (e.g, \cite{Fug74, Jor1, JP99, JT, PW01}) involving unbounded operators are quite different from our present extension problems. Here, our starting point is a fixed partially defined positive definite function $F$. And our present considerations depend primarily on a local analysis. Further, the formulation of our problem, and its solution, relies crucially on a specific reproducing kernel Hilbert space (RKHS) consisting of continuous functions on a fixed subset in $\mathbb{R}^n$. Specifically, we start with a fixed and partially defined (in a proper subset in $\mathbb{R}^n$) function $F$, positive definite and continuous on the subset. From this function $F$ we built the associated RKHS $\mathcal{H}_F$, and we work exclusively in this Hilbert space $\mathcal{H}_F$. Even our assumptions in the statement of our Theorem \ref{MTh}, are stated for a system of operators in this RKHS  $\mathcal{H}_F$. By contrast, the analysis in \cite{Jor1, JP99, JT} is global in nature; and quite different also in a number of other respects. For other literature sources, we mention that our local analysis uses ideas from \cite{AG74} by L. N. Argabright and J. Gil. de Lamadrid, \cite{LS} by Schwartz, and from \cite{NelB} by Nelson. For basic lemmas regarding RKHSs, we refer to \cite{A}, and for unbounded operators, \cite{Sch}.

Returning to positive definite functions, the extension problem for $n=1$ was considered first by Krein \cite{Kn}. For $n > 1$, in the general setting mentioned above, the question of existence 
of positive definite extensions is subtle. For example, for $n = 2$,  Rudin in \cite{Rud1, Rud2} identified cases when no positive definite extensions exist. Other special cases, when extensions are possible, were studied for example in \cite{Jor2,Jor1}. In these results, suitable symmetry restrictions are imposed.

Positive definite functions constitute a big class of positive definite kernels. To see this, let us recall that a function $K\colon X\times X\to\mathbb{C}$, where $X$ is any set, is called a \textit{positive definite kernel} if for any points  $x_1,\ldots,x_m\in X$ and any $c_1,\ldots,c_m\in \mathbb{C}$ we have that
\begin{equation}\label{pdkernel} 
\sum_{j,k=1}^m K(x_j,x_k)c_j\overline{c_k}\geq 0.
\end{equation}
Therefore, we could say that a function $F\colon \Omega-\Omega\to\mathbb{C}$, where $\Omega\subset\mathbb{R}^n$ is an open set, is positive definite if the kernel $K_F\colon \Omega\times\Omega\to\mathbb{C}$ given by $K(x,y)=F(x-y)$ is positive definite. We refer the reader to \cite{A} for an introduction to positive definite kernels. 

Since positive definite kernels are \textit{covariance functions} of Gaussian stochastic processes, results about extensions of positive definite functions have implications for probability theory. Indeed, our present result is 
motivated, in part, by recent developments in the study of Gaussian stochastic processes, see e.g., \cite{AJL, AlJo12}. See also example \ref{enegx} for the Ornstein-Uhlenbeck process. By a classical result \cite{PaSc72}, see also \cite{P,PBook,PV}, every positive definite kernel defined on $S\times S$, where $S$ a fixed set, 
is the covariance function of a Gaussian process indexed by $S$. Therefore, if $S$ is a subset of $\mathbb{R}^n$, a possibility of a positive definite extension to $\mathbb{R}^n$ will then yield a Gaussian process indexed by $\mathbb{R}^n$. And continuity of such an extension is a necessary condition for existence of continuous sample paths for the $\mathbb{R}^n$ process. For applications, 
see e.g., \cite{BGT07}.

While there is a substantial literature on positive definite functions and their role in applications, the problem of \textit{extensions} is less studied. Here we address this question in the case of extension from certain open subsets $\Omega$ in  $\mathbb{R}^n$. For $n = 1$, an extension always exists provided that $\Omega$ is a bounded interval; this was addressed first by Krein \cite{Kn}; but if  $n > 1$, it is known, that extensions may fail to exist, see e.g., \cite{Rud2} and the discussion below. The case of several variables is of interest in physics and in the theory of boundary values (\cite{GJ}). In these areas, the distinction between ``local'' (in some open subset in $\mathbb{R}^n$) vs ``global" (all of $\mathbb{R}^n$) concerns particles at infinity, as well as the presence of \textit{commuting observables}, if we adopt the axiom of quantum theory with observables represented by selfadjoint operators in Hilbert space.
 
For background references on positive definite functions and their applications, see \cite{BCR}, \cite{Fug74}, \cite{Nel}, \cite{CJ}, \cite{BJ}, \cite{PBook}, \cite{P}, \cite{PV}. In a special case, in \cite{Fug74}, the idea of connecting spectral theory with the study of commuting selfadjoint extensions was suggested there.

\subsection{Brief history of the extension problem}

Krein (\cite{Kn}) studies this extension problem in the case when $\Omega$ is a bounded interval $(-r,r)$, $r>0$. He proves that a positive definite extension always exists, but is not necessarily unique. Later Devinatz (\cite{D}) gives necessary conditions for existence of extensions when the domain is a rectangle $(-r_1,r_1)\times (-r_2,r_2)\subset \mathbb{R}^2$. In \cite{Rud1} Rudin extends the result of M. Krein and proves the following. Every radial positive definite function defined on a ball $B(0,r)=\{x\in\mathbb{R}^n\colon |x|<r\}$ extends to a radial positive definite function defined on the entire space $\mathbb{R}^n$. Nussbaum (\cite{N}) gives a characterization (similar to the theorem of Bochner, see section \ref{BT}) of continuous positive definite functions obtaining the result of Rudin as a corollary. Moreover, in \cite{Rud2} the author gives an example of a bounded continuous positive definite function defined on the unit cube $I^n=\{x=(x_1,\ldots,x_n)\in\mathbb{R}^n\colon |x_i|<1,i=1,\ldots,n\}$ that cannot be extended. Therefore, in order to be guaranteed existence of an extension, we have to impose some symmetry conditions on the function, or on the domain of the function. 

\begin{rem} The problem of extending positive definite functions doesn't have to be limited to the case of a subset of Euclidean space. The notion of a positive definite function can be introduced, for example, for groups and semi-groups, and the extension problem can be considered there. One article in this direction that we would like to mention, that we borrowed some ideas from, is \cite{J}, where the author studies positive definite functions defined on a Lie group. 
Moreover, the range of a positive definite function doesn't have to be the set of complex numbers. For example matrix-valued or operator-valued positive definite function have been studied. 
\end{rem}

\subsection{Statement of the problem}

We have the following well-known characterization of continuous positive definite functions (see \cite{St} and reference therein). Let $\Omega\subset\mathbb{R}^n$ be an open set.

\begin{thm}\label{MainLemma} Let $F\colon \Omega-\Omega\to\mathbb{C}$ be a continuous function. Then $F$ is positive definite if and only if for any $\varphi\in C_0(\Omega)$ (see remark below) the following holds
\begin{equation*}
\int_{\Omega}\int_{\Omega}F(y-x)\varphi(x)\overline{\varphi(y)}dxdy\geq 0.
\end{equation*}
\end{thm} 
\begin{rem} The theorem holds true if we replace the class $C_0(\Omega)$ of continuous functions with compact support with the class $C_0^\infty(\Omega)$ of smooth functions with compact support.
\end{rem}

\noindent For $\varphi\in C_0^\infty(\Omega)$ we define a function $F_{\varphi}\colon \Omega\to\mathbb{C}$ by
\begin{equation*}
F_{\varphi}(x)=\int_{\Omega}F(x-y)\varphi(y)dy
\end{equation*} 
and for $\varphi,\psi\in C_0^\infty(\Omega)$ we put
\begin{equation}\label{IP}
\langle F_{\varphi},F_{\psi}\rangle=\int_{\Omega}
\int_{\Omega}F(y-x)\varphi(x)\overline{\psi(y)}dxdy.
\end{equation}
We define 
\begin{equation}\label{densew}
\mathcal{W}_F=\{F_\varphi\colon \varphi\in C^\infty_0(\Omega)\}.
\end{equation}
Then $\mathcal{W}_F$ is a complex vector space and $\langle\cdot,\cdot\rangle$ is a well-defined complex inner-product on $\mathcal{W}_F$. We complete the inner-product space $(\mathcal{W}_F,\langle\cdot,\cdot\rangle)$ to get a complex Hilbert space $\mathcal{H}_F$. For $j=1,\ldots,n$ we consider well-defined densely defined operators $S_j^F\colon \mathcal{W}_F\subset\mathcal{H}_F\to\mathcal{W}_F\subset\mathcal{H}_F$, with common domain $\mathcal{W}_F$, given by
\begin{equation}\label{pdo}
S_j^F(F_\varphi)=-iF_{\frac{\partial\varphi}{\partial x_j}}.
\end{equation} 
The main theorem we prove is the following. 

\begin{thm}\label{MTh} Let $F\colon \Omega-\Omega\to\mathbb{C}$ be a continuous positive definite function, where $\Omega$ is an open and connected subset of $\mathbb{R}^n$. Then there exists an extension of $F$ to a continuous positive definite function defined on the entire Euclidean space $\mathbb{R}^n$ if and only if there exist selfadjoint strongly commuting extension of the system of operators $\{S_j^F\}_{j=1}^n$  in $\mathcal{H}_F$.
\end{thm}

Some preparatory investigations start in section \ref{HS}, and the proof is given in section \ref{ProofMT}. 

\begin{rem} Strong commutativity means that the associated projection-valued measures commute (see section \ref{PVM}).
\end{rem}

Our proof regarding extensions of positive definite functions (Theorem \ref{MTh}) carries over mutatis mutandis to the case of \textit{conditionally negative definite functions}; see section \ref{sec83}. To limit the amount of technical detail in the paper, we have chosen to write out details in full here only for the simpler case of positive definite functions. But the applications to Gaussian processes for the case of conditionally negative definite functions are of far wider reach. In section \ref{sec83}, we note how this fact is of use in applications to harmonic analysis of Gaussian stochastic processes. While stationary Gaussian processes are naturally associated to positive definite functions, there is a much larger family of Gaussian processes whose analysis entails conditionally negative definite functions. This is the class of Gaussian processes \textit{whose increments in mean-square are stationary}. For example, \textit{fractional Brownian motion} is not stationary, but is an example of a process, whose increments in mean-square are stationary; see references \cite{AJL} and \cite{AlJo12} for details.

\subsection{Structure of extensions}

Let $\Omega$ be a non-empty open and connected subset in $\mathbb{R}^n$. Starting with a positive definite and continuous function $F$ on $\Omega -\Omega$, we introduce the compact convex set $\mathcal{K}_F$ of finite positive Borel measures $\mu$ on $\mathbb{R}^n$ as follows: a measure $\mu$ is in $\mathcal{K}_F$  if the Fourier transform of $\mu$ is an extension of $F$ (note that $\mathcal{K}_F$ may be empty). Using a theorem of Nelson, we give necessary and sufficient conditions for $\mathcal{K}_F$ to be a singleton (theorem \ref{tk2} and corollary \ref{cordef}). As a consequence, we get an explicit condition which characterizes the local starting point $(\Omega, F)$ for which we have existence and uniqueness for extension of $F$ to a positive definite continuous function on $\mathbb{R}^n$. In particular, if $\mathcal{K}_F$ contains an element $\mu$ having its support compact in $\mathbb{R}^n$, then $\mathcal{K}_F$ must be a singleton, i.e., $\mathcal{K}_F =\{\mu\}$. 

\subsection{Organization of the paper} 

In section \ref{HS} we study the Hilbert space $\mathcal{H}_F$ (see also \cite{NSX13} where authors consider similar idea). We prove that the form $\langle\cdot,\cdot\rangle$ defined in \eqref{IP} is a complex inner-product on the space $\mathcal{W}_F$ and we show existence of special elements $\gamma_a\in \mathcal{H}_F$, $a\in\Omega$, that behave like Dirac delta functions (see \eqref{rep} and \eqref{reppp}).

We follow with sections, where some well-known results needed for the proof of theorem \ref{MTh} are recalled. We state the classical theorem of Bochner, which says that any positive definite function is the Fourier transform of a finite positive Borel measure (section \ref{BT}), and we gather some facts about orthogonal projection-valued measures (section \ref{PVM}).

In section \ref{PDO} we show that each operator $S_j^F$ in \eqref{pdo} is well-defined, Hermitian, and that its deficiency indices are equal. In particular, it admits selfadjoint extensions, call one such $A_j$. We study the strongly continuous one-parameter group $U_j(t)$ of unitary transformations associated to $A_j$. We show that $U_j$ acts by   translations. 

We follow with section \ref{ProofMT}, where the proof of our main result, which is theorem \ref{MTh}, is provided. 

The remaining sections cover the following topics: (1) We give an example (example \ref{connect}) showing that the assumption of connectedness in the theorem \ref{MTh} cannot be dropped. (2) We study the set of all possible extensions of a given fixed continuous positive definite function. We prove that a positive definite function defined on an open and connected subset of the real line always admits an extension (proven originally by Krein in \cite{Kn}). We give a necessary and sufficient condition for uniqueness of an extension. We also provide an example of a locally defined continuous positive definite function together with two different extensions, which are quiet different from the point of view of Bochner's theorem. (3) We show that the Hilbert space $\mathcal{H}_F$ and \textit{the reproducing kernel Hilbert space} associated to $F$ are unitarily isomorphic via explicit transforms. (4) We prove that if a positive definite function $F$ is extendable, then it is characterized by existence of associated unitary representations of $\mathbb{R}^n$. (5) We briefly talk about the connection of the extension problem with the Fuglede conjecture. (6) We state the relation of the extension problem for positive definite functions with the extension problem for stationary Gaussian stochastic processes. (7) We outline the extension of our main result (theorem \ref{MTh}) to the class of conditionally negative definite functions, which has applications to harmonic analysis of Gaussian stochastic processes, whose increments in mean-square are stationary.

We finish with a summary, where we highlight some main interconnections between theorems in the main body of our paper. We show that a comparison of two distinct extensions may be cast in the language of a \textit{scattering operator} in the sense of Lax-Phillips (\cite{LaPh}).

\section{The Hilbert space $\mathcal{H}_F$}\label{HS}

Let $\Omega\subset\mathbb{R}^n$ be an open set and let $F\colon \Omega-\Omega\to\mathbb{C}$ be a continuous positive definite function. We study families of operators in the Hilbert space $\mathcal{H}_F$ associated to $F$, see also section \ref{sec82}.
Denote by $F^e$ the extension of $F$ to $\mathbb{R}^n$ by putting $F^e=0$ in the complement of $\Omega-\Omega$. 

\begin{rem}
The reader may wonder if $F^e$ is a positive definite function. In general, it is \textit{not}. Indeed, the following fact holds true.
\begin{thm} If $F^e$ is a positive definite function, then $F^e$ is continuous.
\end{thm}
The theorem follows from the fact that if a globally defined positive definite function is continuous at $0$, then it is continuous everywhere (see \cite{BCR} for details) and from the fact that $\Omega-\Omega$ is an open neighborhood of $0$. Since $F^e$ is not continuous, in general, it follows that it is not positive definite. See also example \ref{connect}.
\end{rem}
Then $F_\varphi$, where $\varphi\in C^\infty_0(\Omega)$, is the convolution of $F^e$ with $\varphi$, i.e., for $x\in\Omega$
\begin{equation}\label{convo}
F_\varphi(x)=F^e\star \varphi(x)=\int_{\mathbb{R}^n}F^e(x-y)\varphi(y)dy.
\end{equation}
We note that for any $\varphi\in C^\infty_0(\Omega)$ the function $F^e\star \varphi$ is smooth and it has compact support when $\Omega$ is a bounded set.
Using the fact that $F(x-y)=\overline{F(y-x)}$ for any $x,y\in\Omega$, we obtain for $\varphi,\psi\in C^\infty_0(\Omega)$
\begin{align*}
\int_{\mathbb{R}^n} F_\varphi(x)\overline{\psi(x)}dx &=
\int_{\mathbb{R}^n}
F^e\star \varphi(y)\overline{\psi(y)}dx\\
&=\int_{\mathbb{R}^n}
\int_{\mathbb{R}^n}F^e(x-y)\varphi(y)\overline{\psi(x)}dxdy \\
& =\int_{\mathbb{R}^n}
\int_{\mathbb{R}^n}\overline{F^e(y-x)}\varphi(y)\overline{\psi(x)}dxdy \\
&=\int_{\mathbb{R}^n}\overline{F^e\star \psi (y)}\varphi(y)dy \\
&=\int_{\mathbb{R}^n} \varphi(y)\overline{F_\psi(y)}dy.
\end{align*}
We define forms $R\colon\mathcal{W}_F\times\mathcal{W}_F\to\mathbb{C}$ and $\rho\colon C^\infty_0(\Omega)\times C^\infty_0(\Omega)\to\mathbb{C}$ as follows. For $\varphi,\psi\in C^\infty_0(\Omega)$, we put
\begin{gather}\label{two}
R(F_{\varphi},F_{\psi})=\int_{\mathbb{R}^n} F_\varphi(x)\overline{\psi(x)}dx=\int_{\mathbb{R}^n} \varphi(y)\overline{F_\psi(y)}dy.\\
\label{threee}
\rho(\varphi,\psi)=\int_{\mathbb{R}^n} F_\varphi(x)\overline{\psi(x)}dx.
\end{gather}
We are ready to prove that the form $\langle \cdot,\cdot\rangle \colon \mathcal{W}_F\times\mathcal{W}_F\to\mathbb{C}$ given by \eqref{IP} is well-defined complex inner product. We divide the proof into several steps.

\begin{lem} The form $\langle \cdot,\cdot\rangle \colon \mathcal{W}_F\times\mathcal{W}_F\to\mathbb{C}$ given by \eqref{IP} is well-defined. In other words, the quantity $R(F_\varphi,F_\psi)$ defined by \eqref{two}, or equivalently $\rho(\varphi,\psi)$ defined by \eqref{threee}, is independent of choices of $\varphi$ and $\psi$ (see \eqref{ptwo} and \eqref{pthreee}).
\end{lem}
\begin{proof} Take four functions
$\alpha,\beta,\varphi,\psi\in C^\infty_0(\Omega)$ and suppose that $F_\alpha=F_\varphi$ and $F_\beta=F_\psi$. We want to show that
\begin{equation}\label{ptwo}
R(F_\varphi,F_\psi)=R(F_\alpha,F_\beta),
\end{equation}
or equivalently that
\begin{equation}\label{pthreee}
\rho(\varphi,\psi)=\rho(\alpha,\beta).
\end{equation}
This will follow by \eqref{two}. Indeed,
\begin{align*}
R(F_\varphi,F_\psi) & =\int_{\mathbb{R}^n} F_\varphi(x)\overline{\psi(x)}dx \\
&=\int_{\mathbb{R}^n} F_\alpha(x)\overline{\psi(x)}dx\\
&=\int_{\mathbb{R}^n} \alpha(y)\overline{F_\psi(y)}dy \\
&=\int_{\mathbb{R}^n} \alpha(y)\overline{F_\beta(y)}dy=R(F_{\alpha},F_{\beta}). 
\end{align*}
This finishes the proof.
\end{proof}

\begin{lem} The form $\langle\cdot,\cdot\rangle$ is a complex inner-product on the space $\mathcal{W}_F$. 
\end{lem}
\begin{proof} We need to show that $\langle\cdot,\cdot\rangle$ is positive definite in the strict sense, i.e.,
\begin{equation*}
\langle F_\varphi,F_\varphi\rangle=0\Rightarrow F_\varphi=0.
\end{equation*}
Let $F_\varphi\in \mathcal{W}_F$ be such that
$\langle F_\varphi,F_\varphi\rangle=0$. We take $\lambda\in \mathbb{C}$ and $F_\psi\in \mathcal{W}_F$, and we consider the quantity
\begin{equation*}
\langle F_\varphi+\lambda F_\psi,F_\varphi+\lambda F_\psi\rangle,
\end{equation*}
which we know is non-negative by \eqref{MainLemma}. Now we proceed as in the proof of the Cauchy-Schwarz inequality to get that
\begin{equation*}
\langle F_\varphi,F_\psi\rangle=0.
\end{equation*}
Indeed, we compute
\begin{align*}
0 & \leq \langle F_\varphi+\lambda F_\psi,F_\varphi+\lambda F_\psi\rangle \\
& =2{\rm Re}(\overline{\lambda}\langle F_\varphi,F_\psi\rangle)+|\lambda|^2\langle F_\psi,F_\psi\rangle.
\end{align*}
Taking $\Theta\in [0,2\pi)$ such that $e^{-i\Theta}\langle F_\varphi,F_\psi\rangle\geq 0$ and putting $\lambda=te^{i\Theta}$ for $t\in\mathbb{R}$ we get that
\begin{equation*}
0\leq 2t|\langle F_\varphi,F_\psi\rangle|+t^2\langle F_\psi,F_\psi\rangle.
\end{equation*}
Therefore, since $t\in\mathbb{R}$ is arbitrary, we conclude that $\langle F_\varphi,F_\psi\rangle=0$. 
The fact that $F_\varphi=0$ follows from the lemma below, which is a simple consequence of the continuity of $F$.
\end{proof}

We denote by $(\phi_\epsilon)$ a smoothing kernel, i.e,  we take a smooth non-negative function $\phi$ with support contained in the unit ball $B_1(0)=\{x\in\mathbb{R}^n\colon |x|<1\}$ satisfying $\int\phi=1$, and we define $\phi_\epsilon(x)=1/\epsilon^n\phi(x/\epsilon)$ for $x\in\mathbb{R}^n$. Moreover, for any $a\in\mathbb{R}^n$ we put $\phi_{a,\epsilon}(x)=\phi_\epsilon(x-a)$, where $x\in\mathbb{R}^n$. 
 
\begin{lem}\label{lem1} Put $F_{a,\epsilon}=F_{\phi_{a,\epsilon}}$. For any $a\in \Omega$, we have
\begin{equation}\label{arbiter}
\langle F_\varphi,F_{a,\epsilon}\rangle\xrightarrow[\epsilon \to 0]{} F_\varphi(a).
\end{equation}
\end{lem}

The Hilbert space $\mathcal{H}_F$, as the completion of the space $(\mathcal{W}_F,\langle\cdot,\cdot\rangle)$, consists of equivalence classes of Cauchy sequences of elements from $\mathcal{W}_F$, where the equivalence relation is given by
\begin{equation*}
(F_{\varphi_n})\sim (F_{\psi_n}) \Leftrightarrow \Vert F_{\varphi_n}-F_{\psi_n}\Vert\to 0,
\end{equation*}
where $(F_{\varphi_n})$ and $(F_{\psi_n})$ are Cauchy sequences .
The inner product in $\mathcal{H}_F$ is given by
\begin{equation*}
\langle [(F_{\varphi_n})]_\sim,[(F_{\psi_n})]_\sim\rangle=\lim_{n\to\infty}\langle F_{\varphi_n},F_{\psi_n}\rangle.
\end{equation*}
Moreover, the space $\mathcal{W}_F$ is a dense subset of $\mathcal{H}_F$ under the identification of an element $F_\varphi\in\mathcal{W}_F$ with the constant Cauchy sequence $(F_\varphi)$. Again, by continuity of $F$, we obtain. 

\begin{lem}\label{lem2} For any $a\in\Omega$, the sequence $(F_{a,1/n})$ is Cauchy in $\mathcal{W}_F$. We denote $\gamma_a=[(F_{a,1/n})]_\sim\in\mathcal{H}_F$.
\end{lem}

Combining lemma \ref{lem1} and lemma \ref{lem2} we get that, for any $F_\varphi\in\mathcal{W}_F$ and any $a\in \Omega$,
\begin{equation}\label{rep}
F_\varphi(a)=\langle F_\varphi,\gamma_a\rangle.
\end{equation}
Therefore the $\textrm{span}\,\{\gamma_a\colon a\in\Omega\}$ is a dense subspace of $\mathcal{H}_F$. We also get the following identity
\begin{equation}\label{reppp}
\langle \gamma_b,\gamma_a\rangle=F(a-b),
\end{equation}
where $a,b\in\Omega$. See section \ref{sec82}, where we talk about the reproducing kernel Hilbert space associated to a given positive definite function. 

Starting with a partially defined continuous positive definite function $F$ on a subset  $\Omega-\Omega$ of $\mathbb{R}^n$, as in Theorem \ref{MainLemma}, we then arrive at the reproducing kernel Hilbert space $\mathcal{H}_F$.  Now, this Hilbert space $\mathcal{H}_F$ is defined as a completion, and so by its very nature, it is rather abstract; see our Lemmas 4--7 above.  Nonetheless, $\mathcal{H}_F$ has  a concrete realization as well. It is possible to describe the RKHS  $\mathcal{H}_F$ as a Hilbert space of a definite class of \textit{continuous functions} on $\Omega$. For the details of this, see section \ref{sec82} (in particular, Remark 11 and Lemma 19) where we will need this alternative realization of $\mathcal{H}_F$. 

\section{Positive definite functions on $\mathbb{R}^n$}\label{BT}

We start with the following definitions.
\begin{definition}\label{defi}
\begin{enumerate}
\item[(i)] A positive definite kernel $K$ on $\mathbb{R}^n\times\mathbb{R}^n$, see \eqref{pdkernel}, is said to be \textit{stationary} if there is a function $F\colon\mathbb{R}^n\to\mathbb{C}$ such that 
\begin{equation}
F(x-y)=K(x,y),\quad x,y\in\mathbb{R}^n.
\end{equation}
In other words, a positive definite kernel is stationary if it is given by a positive definite function.
\item[(ii)] A positive definite kernel $K$ on $\mathbb{R}^n\times\mathbb{R}^n$ is said to be \textit{stationary-increment} if there is a function $G\colon\mathbb{R}^n\to\mathbb{C}$ such that
\begin{equation}
G(x-y)=K(x,x)+K(y,y)-2\textrm{Re}\,K(x,y),\quad x,y\in\mathbb{R}^n.
\end{equation} 
In this case, we say that $G$ is a \textit{conditionally negative definite function} defined on $\mathbb{R}^n$. See section \ref{sec83} for additional details.
\item[(iii)] Following \cite{Kle77}, we say that a positive definite function $F\colon\mathbb{R}\to\mathbb{R}$ is \textit{reflection positive} if the kernel
\begin{equation*}
K_{\textrm{ref}}(x,y)=F(x+y),\quad x,y\in [0,\infty),
\end{equation*}
is positive definite.
\end{enumerate}
The function $F(x)=e^{-|x|}$, $x\in\mathbb{R}$, in example \ref{enegx} is reflection positive; it is the covariance function for the Ornstein-Uhlenbeck process (\cite{Kle77}).
\end{definition}
\begin{rem}
Note that if a positive definite kernel $K$ on $\mathbb{R}^n\times\mathbb{R}^n$ is given, then the condition in (i) is more restrictive than the condition in (ii). In other words, every stationary kernel is also stationary-increment.
\end{rem}
Our main theorem (theorem \ref{MTh}) takes on an added significance in view of the theorem of Bochner, characterizing the Fourier transform of finite Borel measures. Specifically, let $M(\mathbb{R}^n)$ denote the set of all finite positive Borel measures on $\mathbb{R}^n$. For a measure $\mu\in M(\mathbb{R}^n)$ we define its Fourier transform $\widehat{\mu}\colon \mathbb{R}^n\to\mathbb{C}$ to be the function
\begin{equation*}
\widehat{\mu}(t)=\int_{\mathbb{R}^n}e^{i t\cdot x}d\mu(x),
\end{equation*}
where $t \cdot x$ denotes the standard inner product on $\mathbb{R}^n$ of $x$ and $t$. The following holds.
\begin{thm} (Bochner) Let $F\colon \mathbb{R}^n\to\mathbb{C}$ be a continuous function. Then $F$ is positive definite if and only if there exists a positive finite Borel measure $\mu\in M(\mathbb{R}^n)$ such that $F=\hat{\mu}$.
\end{thm}
For a proof of we refer the reader to \cite{K}.

\begin{rem} The case of positive definite tempered distributions can be viewed as a generalization of the theorem of Bochner (see \cite{AJL}). Let $T\in S'$ be a tempered distribution. We say that $T$ is positive definite if there exists a positive Borel measure $\nu$ on $\mathbb{R}^n$ satisfying
\begin{equation*}
\int_{\mathbb{R}^n}\frac{d\nu(x)}{(1+|x|^2)^p}<\infty
\end{equation*}
for some $p\in\mathbb{N}$, such that for any test function $\varphi\in S$
\begin{equation*}
T(\varphi)=\int_{\mathbb{R}^n} \widehat{\varphi}(x) d\nu(x),
\end{equation*}
where $\mathcal{S}$ denotes the Schwartz space of rapidly decreasing functions.
\end{rem}

\section{Projection-valued measures}\label{PVM}

Let $(X,\mathcal{B})$ be a measurable space. Let $H$ be a Hilbert space, and let $\textrm{Proj}(H)$ be the set of \textit{orthogonal projections} $P$ on $H$, i.e., $P=P^2=P^\star$. A mapping $E\colon \mathcal{B}\to \textrm{Proj}(H)$ is said to be an \textit{orthogonal projection-valued measure}, often called a \textit{spectral measure}, if for any $f\in H$, the function $\mu_f\colon \mathcal{B}\to [0,\infty)$ given by
\begin{equation*}
\mu_f(\Delta)=\Vert E(\Delta)f\Vert^2=\langle E(\Delta)f,f\rangle=\langle f,E(\Delta)f\rangle,
\end{equation*}
where $E(\Delta)=E(\Delta)^2=E(\Delta)^\star$, is a finite positive measure with
\begin{equation*}
\mu_f(X)=\int_X d\mu_f=\Vert f\Vert^2.
\end{equation*}

\begin{lem}\label{opvm}
Let $H$ be a Hilbert space, let $n>1$, and let $\{A_j\}_{j=1}^n$ be a system of selfadjoint operators (generally unbounded), each with dense domain in $H$, and let $E_j$ be the orthogonal projection-valued measure associated to $A_j$. The measure $E_j$ is defined on the $\sigma$-algebra of Borel sets of the real line, and it recovers the operator $A_j$ via the formula 
\begin{equation*}
A_j=\int_{\mathbb{R}}x\,dE_j(x).
\end{equation*} 
We have that the following conditions are equivalent :
\begin{enumerate}
\item The orthogonal projection-valued measures $E_1,
\ldots,E_n$ commute, i.e., 
\begin{equation*}
E_i(\Delta_1)E_j(\Delta_2)=E_j(\Delta_2)E_i(\Delta_1)
\end{equation*}
for any 
Borel sets $\Delta_1,\Delta_2$ and any $i,j=1,\ldots,n$.
\item For any $z_1,\ldots,z_n\in\mathbb{C}\setminus\mathbb{R}$ the system of bounded operators $\{(z_jI-A_j)^{-1}\}_{j=1}^n$ commutes.
\item For $t=(t_1,\ldots,t_n)\in\mathbb{R}^n$, put 
\begin{align*}
U(t)& =e^{it_1A_1}e^{it_2A_2}\cdot\ldots\cdot e^{it_nA_n} \\
& =\prod_{j=1}^n\int_{\mathbb{R}}e^{it_jx_j}dE_j(x_j).
\end{align*}
Then $\{U(t)\}_{t\in\mathbb{R}^n}$ is a strongly continuous unitary representation of the group $(\mathbb{R}^n,+)$ on $H$.
\item There is an orthogonal projection-valued measure $E$ defined on the Borel $\sigma$-algebra $\mathcal{B}_n=\mathcal{B}_{\mathbb{R}^n}$ such that for all sets $\Delta\in\mathcal{B}_1$
\begin{equation*}
E_j(\Delta)=E(\mathbb{R}\times \ldots\times \Delta\times\ldots\times \mathbb{R}),\quad (\Delta \textrm{ in the $j$-th place}).
\end{equation*}  
If any of the above equivalent conditions is satisfied, then we say that the system $\{A_j\}_{j=1}^n$ of selfadjoint operators strongly commutes, or that the operators $A_1,\ldots,A_n$ strongly commute.
\end{enumerate}
\end{lem}
\begin{proof} See \cite{Jor1},\cite{Jor2},\cite{JT}, \cite{Sch}.
\end{proof}

\section{The operators $S_j^F$}\label{PDO}

For many instances in mathematical physics (see e.g., \cite{Fug74}, \cite{GJ}, \cite{Jor2}, \cite{Nel}), one considers systems of \textit{commuting Hermitian operators} defined on a common dense domain in a fixed Hilbert space. The following problem is an important part of the theory of unbounded linear operators: If each Hermitian operator in the system has selfadjoint extensions, it is natural to ask if, among the individual selefadjoint extensions, we may find a commutative system. It turns out that commuting selfadjoint extensions may not always exist. Our problem about extendability of a fixed locally defined positive definite function $(F, \Omega)$  turns out to be closely connected to existence of commuting selfadjoint extensions for a particular system of commuting Hermitian operators associated with F.  The resolution of the question is an important step in our proof of Theorem \ref{MTh}, and we now turn to the details of this.

\begin{lem} Let $F\colon\Omega-\Omega\to\mathbb{C}$ be a continuous and positive definite function, where $\Omega\subset\mathbb{R}^n$ is an open set.

(i) For any $\varphi,\psi\in C^\infty_0(\Omega)$ the following hold true
\begin{equation}\label{welldefinedoperators}
\langle F_{\frac{\partial \varphi}{\partial x_j}},F_\psi\rangle+\langle F_\varphi,F_{\frac{\partial \psi}{\partial x_j}}\rangle=0,
\end{equation}
where the inner-product is given by \eqref{IP}.

(ii) The operators $S_j^F\colon\mathcal{W}_F\to\mathcal{W}_F$, $j=1,\ldots,n$, from \eqref{pdo} are well-defined.

(iii) The commuting operators $S_j^F\colon\mathcal{W}_F\to\mathcal{W}_F$, $j=1,\ldots,n$, from \eqref{pdo} are Hermitian, i.e., for any $F_\varphi,F_\psi\in\mathcal{W}_F$ we have the following
\begin{equation}
\langle S_j^F(F_\varphi),F_\psi\rangle=\langle F_\varphi,S_j^F(F_\psi)\rangle.
\end{equation}
In particular, $S_j^F\subset (S_j^F)^\star$.

(iv) Each operator $S_j^F\colon\mathcal{W}_F\to\mathcal{W}_F$, $j=1,\ldots,n$, has equal deficiency indices.  
\end{lem}
\begin{proof} (i) We compute using the integration by parts formula
\begin{align*}
\langle F_{\frac{\partial \varphi}{\partial x_j}},F_\psi\rangle &=\int_{\mathbb{R}^n}
\left(F^e\star \frac{\partial\varphi}{\partial x_j}(x)\right)\overline{\psi(x)}dx
\\ & =\int_{\mathbb{R}^n}
\frac{\partial\varphi}{\partial x_j}(F^e\star \varphi(x))\overline{\psi(x)}dx \\
&=-\int_{\mathbb{R}^n}
(F^e\star \varphi(x))\overline{\frac{\partial\psi}{\partial x_j}(x)}dx \\
& =-\langle F_\varphi,F_{\frac{\partial \psi}{\partial x_j}}\rangle.
\end{align*}
Hence (i) follows. 

(ii) To prove that the operators $S_j^F\colon\mathcal{W}_F\to\mathcal{W}_F$ are well-defined we need to show that $F_\varphi=0$ in $\mathcal{H}_F$ implies $F_{\partial \varphi/\partial x_j}=0$ in $\mathcal{H}_F$. This follows from \eqref{welldefinedoperators}. Indeed, let $F_\varphi=0$. Then 
$$\Vert F_{\partial\varphi/\partial x_j}\Vert^2_{\mathcal{H}_F}=
\langle F_{\partial\varphi/\partial x_j},F_{\partial\varphi/\partial x_j}\rangle
=-\langle F_\varphi,F_{\partial^2\varphi/\partial x_j^2}\rangle=0.$$ 

(iii) Follows form \eqref{welldefinedoperators}.  

(iv) Follows from the fact that each operator $S_j^F\colon\mathcal{W}_F\to\mathcal{W}_F$, $j=1,\ldots,n$, commutes with a conjugation, see \cite{Jor2}, \cite{Sch}. 

\end{proof}

Fix $j=1,\ldots,n$. Let $A_j\colon D_j=\textrm{dom}(A_j)\subset \mathcal{H}_F\to\mathcal{H}_F$ be a selfadjoint extension of the operator
$S_j^F\colon \mathcal{W}_F\subset \mathcal{H}_F\to \mathcal{W}_F\subset \mathcal{H}_F$.
Let $E_j$ be the orthogonal projection-valued measure associated to $A_j$, and
let $U_j$ be the strongly continuous one-parameter group of unitary operators defined by $A_j$, i.e.,
\begin{equation*}
U_j(t)=e^{itA_j}=\int_{\mathbb{R}}e^{itx}dE_j(x)
\end{equation*}
for any $t\in\mathbb{R}$. Then $A_j$ is uniquely determined by $U_j$, and for $f\in D_j$ and $t\in \mathbb{R}$ we have that $U_j(t)f\in D_j$. Moreover,
\begin{equation*}
D_j=\{f\in\mathcal{H}_F\colon \frac{d}{dt}|_{t=0}U_j(t)f=\lim_{t\to 0}\frac{U_j(t)f-f}{t}\textrm{ exists}\},
\end{equation*}
and for $f\in D_j$ and $s\in\mathbb{R}$, we have
\begin{equation}\label{aj}
A_jf=-i\frac{d}{dt}|_{t=0}U_j(t)f\quad\textrm{and} \quad \frac{d}{dt}|_{t=s}U_j(t)f=iU_j(s)A_jf=iA_jU_j(s)f.
\end{equation}

If the operators $A_j$ strongly commute, then the product projection-valued measure $E=E_1\times \ldots \times E_n$ on $\mathbb{R}^n$ exists, and the corresponding strongly continuous $n$-parameter group of unitary transformations on $\mathcal{H}_F$ is given by
\begin{equation*}
U(t)=e^{i\sum_{j=1}^nt_jA_j}=\prod_{j=1}^n\int_{\mathbb{R}}e^{it_jx_j}dE_j(x_j) =\int_{\mathbb{R}^n}e^{i t\cdot x}dE(x),
\end{equation*}
where $t=(t_1,\ldots,t_n)\in\mathbb{R}^n$. We will need the following. We denote by $(e_1,\ldots,e_n)$ the standard basis of $\mathbb{R}^n$. 

\begin{lem}
Let $a,b\in\mathbb{R}$ with $a<b$ be such that $[ae_j,be_j]=\{a(1-t)e_j+bte_j\colon 0\leq t\leq 1\}\subset\Omega$. Then 
\begin{equation}\label{translate}
U_j(b-a)\gamma_{ae_j}=\gamma_{be_j}.
\end{equation}
\end{lem}
\begin{proof} Let $a,b\in\mathbb{R}$ be such that $[ae_j,be_j]\subset\Omega$, let $F_\varphi\in \mathcal{W}_F$, and let $(\phi_\epsilon)$ be a smoothing kernel. For each $\epsilon>0$ we consider a function $g_\epsilon:[0,b-a]\to \mathbb{C}$ given by
\begin{equation*}
g(t)=\langle F_\varphi, U_j(b-a-t)F_{(a+t)e_j,\epsilon}\rangle.
\end{equation*}
Then $g$ is a differentiable function with $g'(s)$, $s\in (0,a)$, equal to
\begin{equation}\label{der}
\langle F_\varphi, \left(\frac{d}{dt}|_{t=s}U_j(b-a-t)\right)F_{(a+s)e_j,\epsilon} 
+U_j(b-a-s)\frac{d}{dt}|_{t=s}F_{(a+t)e_j,\epsilon}\rangle.
\end{equation}
Indeed, we just need to show that the derivative 
\begin{equation*}
\frac{d}{dt}|_{t=s}\langle F_\varphi,F_{(a+t)e_j,\epsilon}\rangle 
\end{equation*}
exists. We will show the following
\begin{equation}\label{maineq}
\frac{d}{dt}|_{t=s}\langle F_\varphi,F_{(a+t)e_j,\epsilon}\rangle=\langle F_\varphi,\frac{\partial}{\partial x_j}F_{(a+s)e_j,\epsilon}\rangle.
\end{equation}
We compute
\begin{align*}
& \frac{d}{dt}|_{t=s}\langle F_\varphi,F_{(a+t)e_j,\epsilon}\rangle \\
&=\lim_{t\to s}\frac{1}{t-s}\int_{\Omega}\int_{\Omega} F(y-x)\varphi(x)(\phi_{(a+t)e_j,\epsilon}(y)-\phi_{(a+s)e_j,\epsilon}(y))dxdy.
\end{align*}
Since for any $y\in \Omega$ we have that
\begin{equation*}
\lim_{t\to s}\frac{1}{t-s}(\phi_{(a+t)e_j,\epsilon}(y)-\phi_{(a+s)e_j,\epsilon}(y))=\frac{\partial \phi_{(a+s)e_j,\epsilon}}{\partial x_j}(y),
\end{equation*}
with the help of the dominated convergence theorem, we deduce that 
\begin{align*}
\frac{d}{dt}|_{t=s}\langle F_\varphi,F_{(a+t)e_j,\epsilon}\rangle &=\int_{\Omega}\int_{\Omega} F(y-x)\varphi(x)\frac{\partial \phi_{(a+s)e_j,\epsilon}}{\partial x_j}(y)dxdy \\
&=\langle F_\varphi,F_{\partial \phi_{(a+s)e_j,\epsilon}/\partial x_j}\rangle.
\end{align*}
To get \eqref{maineq} we notice that
\begin{align*}
F_{\partial \phi_{(a+s)e_j,\epsilon}/\partial x_j} & = F^e\star \frac{\partial \phi_{(a+s)e_j,\epsilon}}{\partial x_j} \\
&=\frac{\partial }{\partial x_j}F_{(a+s)e_j,\epsilon}.
\end{align*}
It follows by \eqref{aj}, \eqref{der}, and \eqref{maineq} that $g'(s)=0$ for any $s\in (0,a)$. Indeed, we have
\begin{align*}
& g'(s) \\
& =\langle F_\varphi, \left(\frac{d}{dt}|_{t=s}U_j(b-a-t)\right)F_{(a+s)e_j,\epsilon} 
+U_j(b-a-s)\frac{d}{dt}|_{t=s}F_{(a+t)e_j,\epsilon}\rangle\\
&=\langle F_\varphi,U_j(b-a-s)\left(-\frac{\partial}{\partial x_j}F_{(a+t)e_j,\epsilon}+\frac{d}{dt}|_{t=s}F_{(a+t)e_j,\epsilon}\right)\rangle \\
&=0.
\end{align*}
Therefore
\begin{align*}
\langle F_\varphi,F_{be_j,\epsilon}-U_j(b-a)F_\epsilon\rangle 
& =g(b-a)-g(0)\\
&=0.
\end{align*}
Thus 
\begin{equation*}
\langle F_\varphi,F_{be_j,\epsilon}-U_j(b-a)F_\epsilon\rangle =0,
\end{equation*}
and hence
\begin{equation}\label{epsilon}
F_{be_j,\epsilon}=U_j(b-a)F_{ae_j,\epsilon}.
\end{equation}
We take the limit in \eqref{epsilon} when $\epsilon\to 0$ to obtain \eqref{translate}.
\end{proof}

Let, for each $j=1,\ldots,n$, an operator $A_j\colon D_j=\textrm{dom}(A_j)\subset \mathcal{H}_F\to\mathcal{H}_F$, with $\mathcal{W}_F\subset D_j$, be a selfadjoint extension of the operator $S_j^F\colon \mathcal{W}_F\subset \mathcal{H}_F\to \mathcal{W}_F\subset \mathcal{H}_F$, let $E_j$ be the spectral measure associated to $A_j$, and let $\{U_j(t)\}_{t\in\mathbb{R}}$ be the strongly continuous one-parameter unitary group of operators defined by $A_j$. Suppose that operators $A_1,\ldots,A_n$ strongly commute. This implies that for any $j,k=1,\ldots,n$ and any $s,t\in\mathbb{R}$ the unitary operators $U_j(s)$ and $U_k(t)$ commute, i.e.,
\begin{equation*}
U_j(s)U_k(t)=U_k(t)U_j(s).
\end{equation*}
For each $t=(t_1,\ldots,t_n)\in\mathbb{R}^n$, we define a unitary operator $U(t)$ by
\begin{equation*}
U(t)=U_1(t_1)\cdot\ldots\cdot U_n(t_n).
\end{equation*}
Then for any $c_1,c_2\in\mathbb{R}^n$ we have that
\begin{equation*}
U(c_1+c_2)=U(c_1)U(c_2)=U(c_2)U(c_1).
\end{equation*}
The following fact holds true.

\begin{cor} Let $a,b\in \Omega$ with $\Omega$ connected. Then 
\begin{equation}\label{maintranslation}
U(a-b)\gamma_b=\gamma_a.
\end{equation}
\end{cor}
\begin{proof} Let $c_1,\ldots,c_m\in\Omega$ be such that vectors $a-c_1,c_2-c_1,\ldots,c_m-c_{m-1},b-c_m\in\mathbb{R}^n$ are parallel to the coordinate axes and such that the intervals $$[a,c_1],[c_1,c_2],\ldots,[c_{m-1},c_m],[c_m,b]$$ are contained in $\Omega$. Then, by the previous lemma,
\begin{align*}
U(a-b)\gamma_b=U(a-c_1)U(c_1-c_2)\ldots U(c_m-b)\gamma_b=\gamma_a.
\end{align*}
\end{proof} 

\section{Proof of the Main Theorem}\label{ProofMT}
Beginning with a fixed locally defined positive definite function as in theorem \ref{MTh}, we consider possible positive definite extensions. Let $G\colon \mathbb{R}^n\to\mathbb{C}$ be a continuous positive definite function (if an extension exists). Then by the theorem of Bochner there exists a finite Borel measure $\mu$ such that $G=\hat{\mu}$. For $G_\varphi,G_\psi\in \mathcal{W}_G$, where $\varphi,\psi\in C^\infty_0(\mathbb{R}^n)$, we have
\begin{align*}
\langle G_\varphi,G_\psi\rangle &=\int_{\mathbb{R}^n}\int_{\mathbb{R}^n}G(y-x)\varphi(x)\overline{\psi(y)}dxdy \\
&=\int_{\mathbb{R}^n}\left(\int_{\mathbb{R}^n}e^{-i\langle t,x\rangle}\varphi(x)dx\right)\left(\overline{\int_{\mathbb{R}^n}e^{-i\langle t,y\rangle}\psi(y)dy}\right)d\mu(t).
\end{align*}
We denote for $\varphi\in C^\infty_0(\mathbb{R}^n)$
\begin{equation*}
\widehat{\varphi}(t)=\int_{\mathbb{R}^n}e^{-i\langle t,x\rangle}\varphi(x)dx,\quad
t\in\mathbb{R}^n.
\end{equation*}
Then
\begin{equation}\label{wip}
\langle G_\varphi,G_\psi\rangle=\int_{\mathbb{R}^n}\widehat{\varphi}(t)\overline{\widehat{\psi}(t)}d\mu(t).
\end{equation}
We have proven the following.
\begin{lem}\label{isome}
The linear mapping
\begin{equation}\label{isomdef}
\mathcal{W}_G\ni G_\varphi\mapsto \widehat{\varphi}\in L^2(d\mu)
\end{equation}
is an isometry. Therefore it can be extended by closure to a linear isometry $W\colon\mathcal{H}_G\to L^2(d\mu)$. 
\end{lem}

The following lemma will be applied to positive definite functions defined on $\mathbb{R}^n$ arising as extensions of positive definite functions defined on $\Omega-\Omega$.

\begin{lem}\label{onto} The linear mapping $W\colon\mathcal{H}_G\to L^2(d\mu)$ form lemma \ref{isome} is onto.
\end{lem}
\begin{proof} The range $R(W)$ of the mapping $W$ is a closed subspace of $L^2(d\mu)$, since $W$ is an isometry, by lemma \ref{isome}. We will show that $R(W)$ is a dense subspace of $L^2(d\mu)$, which will finish the proof. Let $f\in L^2(d\mu)$ and suppose that for any $\varphi\in C^\infty_0(\mathbb{R}^n)$
\begin{equation}\label{dense1}
\int_{\mathbb{R}^n}f\widehat{\varphi}d\mu=0.
\end{equation}
We define a tempered distribution $T_f\in \mathcal{S}'$ by
\begin{equation}\label{distribution}
T_f(\varphi)=\int_{\mathbb{R}^n}
f\varphi d\mu,
\end{equation}
where $\varphi\in\mathcal{S}$ (a verification that $T_f$ is indeed a tempered distribution is left to the reader; we recommend \cite{RudFA} for a nice treatment of distribution theory). Then \eqref{dense1} says that the Fourier transform $\widehat{T_f}$ of our distribution $T_f$ is the zero distribution. Therefore $T_f=0$. Hence $\int_{\mathbb{R}^n}f\varphi d\mu=0$ for any $\varphi\in C^\infty_0(\mathbb{R}^n)$. Since $C^\infty_0(\mathbb{R}^n)$ is a dense subspace of $L^2(d\mu)$, $f=0$ as an element of the space $L^2(d\mu)$. Thus the range $R(W)$ of $W$ is a dense subspace of $L^2(d\mu)$.
\end{proof}

\begin{rem}\label{remdist} The tempered distribution $T_f$ in \eqref{distribution} is usually denoted by $fd\mu$.
\end{rem}

\begin{cor}\label{canisom}(See also paragraph \ref{sec82}) Let $G\colon\mathbb{R}^n\to\mathbb{C}$ be a continuous and positive definite function, and let $\mu$ be a finite positive Borel measure on $\mathbb{R}^n$ such that $G=\widehat{\mu}$. Then the Hilbert spaces $\mathcal{H}_G$ and $L^2(d\mu)$ are canonically unitarily isomorphic with the isomorphism given by the mapping $W\colon\mathcal{H}_G\to L^2(d\mu)$ defined by \eqref{isomdef}.   
\end{cor}

\begin{proof} [Proof of theorem \ref{MTh}] $(\Leftarrow)$ Let $F\colon\Omega-\Omega\to\mathbb{C}$ be a continuous and positive definite function. We will construct a continuous positive definite function $G\colon\mathbb{R}^n\to\mathbb{C}$ that extends $F$. By assumption there exist selfadjoint strongly commuting operators $A_j\colon D_j\subset \mathcal{H}_F\to\mathcal{H}_F$ that extend operators $S_j^F\colon \mathcal{W}_F\subset \mathcal{H}_F\to \mathcal{W}_F\subset \mathcal{H}_F$, where $j=1,\ldots,n$. Let $E_j$ be the orthogonal projection-valued measure associated to $A_j$.
We define the following strongly continuous $n$-parameter group of unitary transformations of $\mathcal{H}_F$. For $t=(t_1,\ldots,t_n)\in\mathbb{R}^n$, put
\begin{align*}
U(t) &=e^{i\sum_{j=1}^nt_jA_j}
=\int_{\mathbb{R}^n}e^{i\langle t,x\rangle}dE(x),
\end{align*} 
where $E$ is the product measure (see lemma \ref{opvm}). We fix $x_0\in\Omega$ and we define a function $G\colon\mathbb{R}^n\to\mathbb{C}$ by
\begin{equation}\label{exten}
G(t)=\langle \gamma_{x_0},U(t)\gamma_{x_0}\rangle.
\end{equation}
Therefore $G$ is a continuous positive definite function. We will show that $G$ is an extension of our function $F$. For $a-b\in\Omega-\Omega$ we obtain
\begin{align*}\label{finalcomp}
G(a-b) & =\langle \gamma_{x_0},U(a-b)\gamma_{x_0}\rangle \\
&=\langle U(b)\gamma_{x_0},U(a)\gamma_{x_0}\rangle \\
& =\langle U(x_0)U(b-x_0)\gamma_{x_0},U(x_0)U(x_0-a)\gamma_{x_0}\rangle \\
&=\langle U(x_0)\gamma_b,U(x_0)\gamma_a\rangle \\
&=\langle \gamma_b,\gamma_a\rangle \\
&=F(a-b).
\end{align*} 
This proves that $G\colon\mathbb{R}^n\to\mathbb{C}$ is a continuous positive definite extension of $F$.

($\Rightarrow$) Let $F\colon\Omega-\Omega\to\mathbb{C}$ be a continuous positive definite function, and let $G\colon \mathbb{R}^n\to\mathbb{C}$ be a continuous positive definite function that extends $F$. Let, by the theorem of Bochner, $\mu$ be a finite Borel measure such that $G=\hat{\mu}$.

Below we shall work in the Hilbert space $\mathcal{H}_F$ and the inner product $\langle\cdot,\cdot\rangle=\langle\cdot,\cdot\rangle_{\mathcal{H}_F}$ will be understood to be that of the fixed reproducing kernel Hilbert space $\mathcal{H}_F$.

\textit{Step 1}. There exists a unique projection-valued measure $E\colon (\mathbb{R}^n,\mathcal{B}_n)\to\textrm{Proj}(\mathcal{H}_F)$ such that
\begin{equation}\label{extpvm}
\langle F_\varphi,E(B)F_\psi\rangle=\int_B\widehat{\varphi}(t)\overline{\widehat{\psi}(t)}d\mu(t),
\end{equation}
where $B\in \mathcal{B}_n$ is a Borel subset of $\mathbb{R}^n$ and $F_\varphi,F_\psi\in \mathcal{W}_F$, i.e., $\varphi,\psi\in C^\infty_0(\Omega)$.

The construction of such a projection-valued measure is given in \cite{NelB} and \cite{RudFA}, therefore we just sketch the proof. The uniqueness follows from density of $\mathcal{W}_F=\{F_\varphi\in C^\infty_0(\Omega)\}$ in $\mathcal{H}_F$. Moreover, for any $F_\varphi,F_\psi\in \mathcal{W}_F$ and any Borel set $B\subset\mathbb{R}^n$, by Cauchy-Schwarz,
\begin{align*}
\left| \int_B\widehat{\varphi}(t)\overline{\widehat{\psi}(t)}d\mu(t) \right|^2 &
\leq \int_B|\widehat{\varphi}(t)|^2d\mu(t)\int_B|\widehat{\psi}(t)|^2d\mu(t) \\
& \leq \int_{\mathbb{R}^n}|\widehat{\varphi}(t)|^2d\mu(t)\int_{\mathbb{R}^n}|\widehat{\psi}(t)|^2d\mu(t) \\
& \leq \Vert F_\varphi\Vert_{\mathcal{H}_F}^2\Vert F_\psi\Vert_{\mathcal{H}_F}^2.
\end{align*}
Therefore \eqref{extpvm} determines a unique bounded selfadjoint orthogonal projection $E(B)\colon \mathcal{H}_F\to\mathcal{H}_F$. To show sigma-additivity, let $\{B_k\}_{k\in\mathbb{N}}$ be a collection of disjoint Borel sets. Then for $F_\varphi\in \mathcal{W}_F$
\begin{align*}
\langle F_\varphi,E(\cup_{k=1}^\infty B_k)F_\varphi\rangle &=\int_{\cup_{k}B_k}|\widehat{\varphi}(t)|^2d\mu(t)\\
&=\sum_{k=1}^\infty\int_{B_k}|\widehat{\varphi}(t)|^2d\mu(t)\\
&=\sum_{k=1}^\infty \langle F_\varphi,E(B_k)F_\varphi\rangle,
\end{align*}
so $E(\cup_{k=1}^\infty B_k)=\sum_{k=1}^\infty E(B_k)$. Finally,
\begin{equation*}
\langle F_\varphi,E(\mathbb{R}^n)F_\varphi\rangle=\int_{\mathbb{R}^n}|\widehat{\varphi}(t)|^2d\mu(t)=\Vert F_\varphi\Vert^2_{\mathcal{H}_F}.
\end{equation*}

\textit{Step 2}. By the theorem of Stone (about one-to-one correspondence between orthogonal projection-valued measures and $n$-parameter strongly continuous group of unitary transformations, see \cite{NelB},\cite{RudFA}), the collection of operators
$\{U(t)\colon \mathcal{H}_F\to\mathcal{H}_F\}_{t\in\mathbb{R}^n}$ given by
\begin{equation}\label{stone1to1}
U(t)=\int_{\mathbb{R}^n}e^{it\cdot x}dE(x),
\end{equation}  
where $E$ is our orthogonal projection-valued measure from step 1, defines an $n$-parameter strongly continuous group of unitary transformations of $\mathcal{H}_F$. For $F_\varphi\in \mathcal{W}_F$ we define a finite Borel measure $\mu_{F_\varphi}$ on $\mathbb{R}^n$ by
\begin{equation*}
\mu_{F_\varphi}(B)=\langle F_\varphi, E(B)F_\varphi\rangle.
\end{equation*}
Then, by \eqref{extpvm}, $d\mu_{F_\varphi}=|\widehat{\varphi}|^2d\mu$ in the sense of Radon-Nikodym derivatives. In particular, for $F_\varphi\in \mathcal{W}_F$ we obtain
\begin{align*}
\langle F_\varphi,U(t)F_\varphi\rangle &=\int_{\mathbb{R}^n}e^{it\cdot x}d\mu_{F_\varphi}(x)\\
&=\int_{\mathbb{R}^n}e^{it\cdot x}|\widehat{\varphi}(x)|^2d\mu(x).
\end{align*}

\textrm{Step 3}. The unitary representation $\{U(t)\colon\mathcal{H}_F\to\mathcal{H}_F\}_{t\in\mathbb{R}^n}$ determines a collection $\{A_j\colon D_j\subset \mathcal{H}_F\to\mathcal{H}_F\}_{j=1}^n$ of strongly commuting selfadjoint operators in $\mathcal{H}_F$ as follows. We put
\begin{equation*}
D_j=\{f\in\mathcal{H}_F\colon \frac{\partial }{\partial t_j}|_{t=0}U(t)f=\lim_{h\to 0}\frac{U(he_j)f-f}{h}\textrm{ exists}\}
\end{equation*} 
and for $f\in D_j$
\begin{equation*}
A_jf=-i\frac{\partial }{\partial t_j}|_{t=0}U(t)f.
\end{equation*}
Equivalently, using our orthogonal projection-valued measure $E$, we have  
\begin{equation*}
D_j=\{f\in\mathcal{H}_F\colon \int_{\mathbb{R}^n} |x_j|^2d\mu_f(x)<\infty\},
\end{equation*}
where $\mu_f$ is a finite Borel measure define by $\mu_f(B)=\langle f,E(B)f\rangle$ for a Borel set $B\subset\mathbb{R}^n$, and
\begin{equation*}
A_j=\int_{\mathbb{R}^n}x_jdE(x).
\end{equation*}

\textit{Step 4}. It remains to show that operators $A_j$ extend operators $S_j^F$. We first show that $\mathcal{W}_F\subset D_j$. Let $F_\varphi\in \mathcal{W}_F$. Then, since $ix_j\widehat{\varphi}(x)=\widehat{\frac{\partial \varphi}{\partial x_j}}(x)$,
\begin{align*}
\int_{\mathbb{R}^n}|x_j|^2d\mu_{\varphi}(x)& =\int_{\mathbb{R}^n}|x_j\widehat{\varphi}(x)|^2d\mu(x)\\
&=\int_{\mathbb{R}^n}|\widehat{\frac{\partial\varphi}{\partial{x_j}}}(t)|^2d\mu(x)<\infty.
\end{align*} 
Thus $F_\varphi\in D_j$. Moreover, for any $F_\varphi\in \mathcal{W}_F$
\begin{align*}
\langle A_jF_\varphi,F_\varphi\rangle &
=\int_{\mathbb{R}^n}x_jd\mu_{F_\varphi}(x) \\
&=\int_{\mathbb{R}^n}x_j|\widehat{\varphi}(x)|^2d\mu(x)\\
&=\int_{\mathbb{R}^n}-i\widehat{\frac{\partial \varphi}{\partial x_j}}(x)\overline{\widehat{\varphi}(x)}d\mu(x)\\
&=\langle -iF_{\frac{\partial\varphi}{\partial x_j}},F_\varphi\rangle\\
&=\langle S_j^F(F_\varphi),F_\varphi\rangle.
\end{align*}
This proves the theorem.

\end{proof}

\begin{rem} We note that, by \eqref{maintranslation}, it follows that the extension given in \eqref{exten} is independent of the choice of a point $x_0\in\Omega$.
\end{rem}

\section{Applications and Examples}

\subsection{On the connectedness assumption in theorem \ref{MTh}}

We give an example of a continuous positive definite function defined on an open, but not connected, subset of $\mathbb{R}$ that cannot be extended to a continuous positive definite function defined on the entire $\mathbb{R}$. This shows that we cannot drop the assumption of connectedness of a set $\Omega$ in the statement of theorem \ref{MTh}.

\begin{ex}\label{connect} 
We put $\Omega=(-\frac{1}{4},\frac{1}{4})\cup (\frac{3}{4},1)$.
We also denote $\Omega_A=(-\frac{1}{4},\frac{1}{4})$ and $\Omega_B=(\frac{3}{4},1)$. 
Then one can easily show that $\Omega-\Omega=(-\frac{5}{4},-\frac{1}{2})\cup (-\frac{1}{2},\frac{1}{2})\cup (\frac{1}{2},\frac{5}{4})$. We define a function $F\colon\Omega-\Omega\to [0,1]$ as follows.
\begin{equation*}
F(x)=
\begin{cases}
1-|x|, & x\in (-\frac{1}{2},\frac{1}{2}) \\
0, & x\in (-\frac{5}{4},-\frac{1}{2})\cup (\frac{1}{2},\frac{5}{4}).
\end{cases}
\end{equation*}
We will show that $F$ is a positive definite function. Let $x_1,\ldots,x_m\in\Omega$ and let $c_1,\ldots,c_m\in\mathbb{C}$. Then
\begin{align*}
\sum_{j,k=1}^m F(x_j-x_k)c_j\overline{c_k} =\sum_{x_j,x_k\in \Omega_A}F(x_j-x_k)c_j\overline{c_k}+\sum_{x_j,x_k\in\Omega_B} F(x_j-x_k)c_j\overline{c_k}
\end{align*}
Both summations above are nonnegative, since the function 
\begin{equation*}
\mathbb{R}\ni x\to \max(0,1-|x|)
\end{equation*}
is positive definite, which follows by a computation, or by the following Polya's criterion for characteristic functions (see \cite{Pol}, \cite{L}). Clearly $F$ has no extension to a continuous positive definite function defined on the entire real line.
\begin{thm}(Polya)
Let $f\colon [0,\infty)\to [0,\infty)$ be a continuous, decreasing, and convex function such that $\lim_{x\to \infty}f(x)=0$. Then the even extension $g\colon\mathbb{R}\to [0,\infty)$ of $f$, i.e., $g(x)=f(|x|)$ for $x\in \mathbb{R}$, is a positive definite function. Moreover, if $\lambda$ is a measure such that $\widehat{\lambda}=g$, then $\lambda$ is absolutely continuous with respect to the Lebesgue measure. If $d\lambda=Adt$ is the Radon-Nikodym derivative of $\lambda$ with respect to the Lebesgue measure, where $A\in L^1(\mathbb{R})$, then
\begin{equation*}
A(t)=\frac{1}{2\pi}\int_{\mathbb{R}}e^{itx}g(x)dx.
\end{equation*} 
\end{thm}
\end{ex}

\subsection{The reproducing kernel Hilbert space associated to a given positive definite function}\label{sec82}

There is a construction that associates a reproducing kernel Hilbert space to a positive definite kernel (see \cite{A}). We sketch the construction below in the case of a positive definite function. We recall the definition of a reproducing kernel Hilbert space (see (a) below), we follow with the construction (see (b) below), and finally we talk about relation with our space $\mathcal{H}_F$ (see (c) below). 

(a) Let $X$ be a set. A Hilbert space $\mathcal{H}$ is called a \textit{reproducing kernel Hilbert space} on $X$  if 
\begin{enumerate}
\item the elements of $\mathcal{H}$ are functions on $X$,
\item for each $x\in X$ there exists a unique element $k_x\in\mathcal{H}$ such that for any $f\in \mathcal{H}$ the following reproducing identity holds true:
\begin{equation}\label{repkerhilsp}
f(x)=\langle f,k_x\rangle_{\mathcal{H}}.
\end{equation}
\end{enumerate}
The mapping $K\colon X\times X\to\mathbb{C}$ given by $K(x,y)=\langle k_y,k_x\rangle_{\mathcal{H}}$ is called the \textit{reproducing kernel} of $\mathcal{H}$. We note that if $X$ is a topological space and if $K$ is continuous, then elements of the reproducing kernel Hilbert space $\mathcal{H}$ are continuous functions, and the mapping $X\ni x\to k_x\in \mathcal{H}$ is continuous.
 
(b) Let $F\colon\Omega-\Omega\to\mathbb{C}$ be a positive definite function, where $\Omega$ is an open subset of $\mathbb{R}^n$. We associate a reproducing kernel Hilbert space to it in the following way. For each $a\in\Omega$ we define a function $k_a\colon\Omega\to\mathbb{R}$ by
\begin{equation}\label{ka}
k_a(x)=F(x-a).
\end{equation}
Then the space $W_F=\textrm{span}\,\{k_a\colon a\in\Omega\}$ together with the following sesquilinear form
\begin{equation}\label{rkhsform}
\langle \sum_{j=1}^mc_jk_{a_j},\sum_{j=1}^md_jk_{a_j}\rangle_F=\sum_{j,k=1}^m F(a_k-a_j)c_j\overline{d_k}
\end{equation}
is a complex inner-product space. Its completion is a Hilbert space, which we will realize as a reproducing kernel Hilbert space,  called the reproducing kernel Hilbert space associated to $F$. By \eqref{rkhsform} for any $f\in W_F$ and any $a\in \Omega$ we obtain
\begin{equation*}
f(a)=\langle f,k_a\rangle_F.
\end{equation*}
If $(f_n)$ is a Cauchy sequence in $W_F$ and if $a\in \Omega$, then for any $m,n\in\mathbb{N}$ we have
\begin{equation*}
|f_m(a)-f_n(a)|=|\langle f_m-f_n,k_a\rangle_F|\leq \Vert f_m-f_n\Vert_F\Vert k_a\Vert_F.
\end{equation*}
Therefore $(f_n(a))$ is a Cauchy sequence in $\mathbb{C}$, and as such, it converges to some complex number $f(a)$. We have assigned a complex-valued function $\Omega\ni a\mapsto f(a)\in\mathbb{C}$ to the Cauchy sequence $(f_n)$. This assignment with the inner product induced from the completion of $W_F$ gives as a reproducing kernel Hilbert space $H_F$, where the constant Cauchy sequence $(f)$, $f\in W_F$, is mapped to $f$. We note that
\begin{equation*}
\langle k_a,k_b\rangle_F=F(b-a).
\end{equation*}

(c) We can apply the above argument to our Hilbert space $\mathcal{H}_F$, and realize it as a reproducing kernel Hilbert space, which in view of \eqref{rep} is exactly the space $H_F$; and the Cauchy sequence $\gamma_a$ is represented by the function $k_a$, $a\in\Omega$. To distinguish between spaces $\mathcal{H}_F$ and $H_F$ we denote by $V\colon \mathcal{H}_F\to H_F$ the isomorphism (identification) given by
the assignment
\begin{equation*}
\gamma_a\mapsto k_a.
\end{equation*}

Let $F\colon\mathbb{R}^n\to\mathbb{C}$ be a continuous positive definite function, and let $\mu$ be a finite positive Borel measure on $\mathbb{R}^n$ such that $F=\widehat{\mu}$. Then, by corollary \ref{canisom} and the above discussion, Hilbert spaces $\mathcal{H}_F$, $H_F$, and $L^2(d\mu)$ are unitarily isomorphic. 
We explicitly write down the unitary isomorphism $Z\colon L^2(d\mu)\to H_F$ such that $V=Z\circ W$. Hence we have the following diagram of intertwining operators with arrows representing unitary (isometric) isomorphisms.
\begin{equation}\label{diagram}
\xymatrix{\mathcal{H}_F \ar[d]^V \ar[r]^W  & L^2(d\mu)\ar[dl]_Z\\
H_F
}
\end{equation}
We define $Z\colon L^2(d\mu)\to H_F$ as follows. For $f\in  L^2(d\mu)$, we put 
\begin{equation}
(Zf)(x)=\int_{\mathbb{R}^n}e^{it\cdot x}f(t)d\mu(t)\quad (=\widehat{fd\mu}(x)).
\end{equation}
(see remark \ref{remdist}). We need to show: 1. $Zf\in H_F$, 2. $Z\colon L^2(d\mu)\to H_F$ is a unitary isomorphism, 3. $V=Z\circ W$. We only prove that $Zf\in H_F$, and leave the proofs of 2. and 3. to the reader. We use the following fact (see \cite{A}). 
\begin{lem}\label{inrkhs} A function $g\colon\mathbb{R}^n\to\mathbb{C}$ belongs to the reproducing kernel Hilbert space $H_F$
if and only if there exists a positive constant $C>0$ such that the mapping \begin{equation}\label{elem}
\mathbb{R}^n\times\mathbb{R}^n\ni (x,y)\mapsto C^2F(x-y)-g(x)\overline{g(y)}\in\mathbb{C}
\end{equation}
is a positive definite kernel. 
\end{lem} 

\begin{rem} When our continuous positive function $F$ is defined on the set $\Omega-\Omega$, the above lemma still holds true with $\mathbb{R}^n$ replaced by $\Omega$. Moreover, we can rewrite the lemma using an integral conditions as in theorem \ref{MainLemma}. More precisely, a continuous function $g\colon \Omega\to\mathbb{C}$ belongs to the space $H_F$ if and only if there exists a constant $C>0$ such that
\begin{equation*}
|\int_{\Omega}g(x)\varphi(x)dx|^2\leq C\int_\Omega\int_\Omega F(y-x)\varphi(x)\overline{\varphi(y)}dxdy
\end{equation*}
for any $\varphi\in C^\infty_0(\Omega)$.
\end{rem} 

We establish that the mapping in \eqref{elem} with $g=Zf$, where $f\in L^2(d\mu)$, and with $C=\Vert f\Vert_{L^2(d\mu)}$ is a positive definite kernel, so $\widehat{fd\mu}\in H_F$. Let $x_1,\ldots,x_m\in\mathbb{R}^n$ and let $c_1,\ldots,c_m\in\mathbb{C}$. Then
\begin{equation*}
\sum_{j,k=1}^mc_j\overline{c_k}F(x_j-x_k)=\int_{\mathbb{R}^n} \left|\sum_{j=1}^m c_je^{it\cdot x_j}\right|^2d\mu(t),
\end{equation*}
and
\begin{equation*}
\sum_{j,k=1}^mc_j\overline{c_k}(Zf)(x_j)\overline{(Zf)(x_k)}=\left|\int_{\mathbb{R}^n} \sum_{j=1}^m c_je^{it\cdot x_j}f(t)d\mu(t)\right|^2.
\end{equation*}
Therefore, by Cauchy-Schwarz inequality,
\begin{equation*}
\sum_{j,k=1}^mc_j\overline{c_k}[C^2F(x_j-x_k)-(Zf)(x_j)\overline{(Zf)(x_k)}]\geq 0.
\end{equation*}
Thus $g=Zf\in H_F$. We note that since $f\in L^2(d\mu)\subset L^1(d\mu)$, it follows that $Zf=\widehat{fd\mu}$ is a continuous function.

Let $F\colon\Omega-\Omega\to\mathbb{C}$ be a continuous and positive definite function, let $G\colon \mathbb{R}^n\to\mathbb{C}$ be a continuous and positive definite function that extends $F$, and let $\mu$ be a finite positive Borel measure such that $\widehat{\mu}=G$. We note that for any $\varphi\in C^\infty_0(\Omega)$ the function $G_\varphi\colon\mathbb{R}^n\to\mathbb{C}$ extends the function $F_\varphi\colon \Omega\to\mathbb{C}$ and $\Vert F_\varphi\Vert_{\mathcal{H}_F}=\Vert G_\varphi\Vert_{\mathcal{H}_G}$. Therefore, the mapping $\mathcal{W}_F\ni F_\varphi\mapsto G_\varphi\in \mathcal{W}_G$ extends by closure to an isometric embedding from $\mathcal{H}_F$ to $\mathcal{H}_G$. Hence, we can view $\mathcal{H}_F$ as a closed subspace of $\mathcal{H}_G$. Similar remark hold for the spaces $H_F$ and $H_G$.
\begin{cor} We have the following decomposition
\begin{equation*}
L^2(d\mu)=W(\mathcal{H}_F)\oplus \{f\in L^2(d\mu)\colon \widehat{fd\mu}=0 \textrm{ on }\Omega\}.
\end{equation*}
\end{cor} 
\begin{proof} Follows from the above discussion. With the identification of spaces $\mathcal{H}_F$ with $H_F$ and $\mathcal{H}_G$ with $H_G$, setting $T\colon \mathcal{H}_F\to L^2(d\mu)$ by $T(F_\varphi)=\widehat{\varphi}$, where $\varphi\in C^\infty_0(\Omega)$, we get that $T$ is an isometry and $T^\star f=\widehat{fd\mu}|_{\Omega}$ for all $f\in L^2(d\mu)$.
\end{proof}

\begin{rem}
There are examples with $\dim \mathcal{H}_F=1$ and $\dim L^2(d\mu)=\infty$.
\end{rem}

The following property holds for any connected open set $\Omega$ in $\mathbb{R}^n$, and any continuous extendable positive definite function $F$ on   $\Omega-\Omega$. The conclusion of the corollary is an \textit{interpolation} property for the pair $(\Omega, F)$. 

\begin{cor}\label{cor27} Let $\Omega\subset\mathbb{R}^n$ be a non-empty, open, and connected set, and let $F\colon\Omega-\Omega\to\mathbb{C}$ be a continuous positive definite function. Assume that $F$ has a continuous positive definite extension $\tilde{F}\colon\mathbb{R}^n\to\mathbb{C}$; see theorem \ref{MTh}. Then for every $x\in \mathbb{R}^n$ the function $\tilde{F}(\cdot-x)\colon\Omega\to\mathbb{C}$ belongs to the reproducing kernel Hilbert space $H_F$ of functions on $\Omega$.  
\end{cor} 
\begin{proof} By theorem \ref{MTh} and corollary \ref{canisom}, there exists a strongly continuous unitary representation $U$ of $\mathbb{R}^n$ on $H_F$ such that if we fix $a\in\Omega$, then
\begin{equation}
\tilde{F}(x)=\langle k_a,U(x)k_a\rangle_F,\quad x\in \mathbb{R}^n.
\end{equation}
To show that 
\begin{equation}\label{fmx}
\tilde{F}(\cdot-x)\in H_F,
\end{equation}
where $x\in\mathbb{R}^n$, we use lemma \ref{inrkhs}. Without loss of generality we may assume that $F(0)=1$. We prove that for any $y_1,\ldots,y_m\in \Omega$, and $c_1,\ldots,c_m\in\mathbb{C}$, and $x\in\mathbb{R}^n$ we have the following estimate:
\begin{equation}\label{est}
\sum_{j,k=1}^m c_j\overline{c_k}(F(y_j-y_k)-\tilde{F}(y_j-x)\overline{\tilde{F}(y_k-x)})\geq 0.
\end{equation} 
The desired conclusion \eqref{fmx} will follow from lemma \ref{inrkhs}. We show that \eqref{est} holds true:
\begin{align*}
\sum_{j,k=1}^m c_j\overline{c_k}\tilde{F}(y_j-x)\overline{\tilde{F}(y_k-x)} 
& =|\langle U(x)k_a,\sum_{j=1}^mc_jU(y_j)k_a\rangle_F|^2\\
& \leq \Vert U(x)k_a\Vert^2_F\Vert \sum_{j=1}^mc_jU(y_j)k_a\Vert^2_F\\
& =F(0)\sum_{j,k=1}^mc_j\overline{c_k}F(y_j-y_k),
\end{align*}
since we have normalized $F(0)=1$.
\end{proof}

\subsection{The set of all possible extensions}

We start with the following corollary to the proof of theorem \ref{MTh} relating existence of an extension of a positive definite function $F$ to existence of a unitary representation of $\mathbb{R}^n$ on the Hilbert space $\mathcal{H}_F$. 

Let $F\colon\Omega-\Omega\to\mathbb{C}$ be a continuous positive definite function, and let $G\colon \mathbb{R}^n\to\mathbb{C}$ be a continuous positive definite function that extends $F$. We have the Hilbert space $\mathcal{H}_F$ associated to the positive definite function $F\colon \Omega-\Omega\to\mathbb{C}$, and the Hilbert space $\mathcal{H}_G$ associated to the positive definite function $G\colon \mathbb{R}^n\to\mathbb{C}$. Moreover, we have dense subsets $\mathcal{W}_F\subset\mathcal{H}_F$ and $\mathcal{W}_G\subset\mathcal{H}_G$ defined in \eqref{densew}. Let, by the theorem of Bochner, $\mu$ be a finite Borel measure such that $G=\hat{\mu}$.
We define a representation $\{V_\mu(t)\}_{t\in\mathbb{R}^n}$ of $\mathbb{R}^n$ on $L^2(d\mu)$ by
\begin{equation}\label{repmu}
(V_\mu(t)f)(x)=e_t(x)f(x),
\end{equation}
where $f\in L^2(d\mu)$, $x,t\in\mathbb{R}^n$, and 
\begin{equation}\label{spec}
e_t(x)=e^{i\sum_{j=1}^n t_jx_j}=e^{it\cdot x}.
\end{equation}

\begin{cor}\label{intert} Let $\Omega\subset\mathbb{R}^n$ be an open and connected set and let $F\colon\Omega-\Omega\to\mathbb{C}$ be a continuous positive definite function. Then $F$ is extendable to a continuous positive definite function $G$ defined on  $\mathbb{R}^n$ if and only if there exists a unitary representation $\{U(t)\}_{t\in\mathbb{R}^n}$ on $\mathcal{H}_F$ such that $(U(t),\mathcal{H}_F)$ is unitarily equivalent to the representation $(V_\mu,L^2(d\mu))$, where $\mu$ is a positive measure such that $G=\widehat{\mu}$. 
\end{cor}
\begin{proof} This corollary is a consequence of what we did in the proof of theorem \ref{MTh} and spectral analysis of Stone; it is nicely done in the book \cite{NelB} by Nelson. Let $E$ be the orthogonal projection-valued measure from the proof of theorem \ref{MTh}. Let $x_0\in\Omega$ and let $F_0=\gamma_{x_0}\in \mathcal{H}_F$. Recall that $\mu$ is the measure $\mu_{F_0}$ given by $\mu_{F_0}(B)=\langle F_0,E(B)F_0\rangle=\Vert E(B)F_0\Vert^2_{\mathcal{H}_F}$, where $B\subset\mathbb{R}^n$ is a Borel subset of $\mathbb{R}^n$, i.e., $G(t)=\widehat{\mu}(t)=\langle F_0,U(t)F_0\rangle$, where $\{U(t)\}_{t\in\mathbb{R}^n}$ is the representation $U(t)=\int e^{it\cdot x}dE(x)$. For any $f\in L^2(d\mu)$ we have an element $f(A_1,\ldots,A_n)F_0\in\mathcal{H}_F$ given by
\begin{equation*}
f(A_1,\ldots,A_n)F_0=\left(\int_{\mathbb{R}^n} f(x_1,\ldots,x_n)dE(x)\right)F_0.
\end{equation*}
Moreover,
\begin{align*}
\Vert f(A_1,\ldots,A_n)F_0\Vert^2_{\mathcal{H}_F}=\int_{\mathbb{R}^n}|f(x)|^2d\mu(x).
\end{align*}
Therefore the mapping $W\colon \mathcal{H}_F\to L^2(d\mu)$ defined by $W(f(A_1,\ldots,A_n)F_0)=f$, for $f\in L^2(d\mu)$, is an unitary isomorphism. From the construction, $W$ intertwines representations $U(t)$ and $V_\mu(t)$, i.e., $WU(t)=V_\mu(t)W$. Indeed, $$U(t)f(A_1,\ldots,A_n)F_0=g_t(A_1,\ldots,A_n),$$ where $g_t(x_1,\ldots,x_n)=e^{it\cdot x}f(x_1,\ldots,x_n)$. This finishes the proof.

\end{proof}

Using the explicit unitary isomorphisms of Hilbert spaces $\mathcal{H}_F$, $H_F$, we can rewrite the above corollary with $\mathcal{H}_F$ replaced by $H_F$.

We now prove, using theorem \ref{MTh}, that when $n=1$ an extension always exists. This is the original result of Krein (\cite{Kn}).

\begin{cor}\label{corkrein}$(n=1)$ A continuous positive definite function $F\colon\Omega-\Omega\to\mathbb{C}$, with $\Omega=(p,r)\subset \mathbb{R}$,  being a bounded and open interval, always has an extension to a continuous positive definite function defined on the entire real line $\mathbb{R}$. Moreover, there is either precisely one, or at least a one-parameter family of distinct continuous positive definite extensions.
\end{cor}
\begin{proof} We need to show that the operator
$S^F\colon \mathcal{W}_F\subset \mathcal{H}_F\to\mathcal{W}_F\subset\mathcal{H}_F$, $S^F(F_\varphi)
=-iF_{\varphi'}$,
has a selfadjoint extension. We define $K\colon (p,r)\to (p,r)$ by $K(t)=-t+p+r$. For $\varphi\in C^\infty_0(\Omega)$ put $\varphi_K\colon \Omega\to\mathbb{C}$ by $\varphi_K(t)=\overline{\varphi(K(t))}$, and then set $JF_{\varphi}=F_{\varphi_K}$. Then $J\colon\mathcal{W}\to\mathcal{W}$ is a conjugation, i.e,. $J$ is conjugate-linear $J(iF_\varphi)=-iJF_\varphi$, $J^2={\rm Id}$, and $\Vert JF_\varphi\Vert=\Vert F_\varphi\Vert$ for any $F_\varphi\in \mathcal{W}_F$. 
Moreover $J\circ(S^F)=(S^F)\circ J$. This follows from the fact that $J$ is conjugate-linear and that $\frac{d \varphi_K}{dx}=-(\frac{d \varphi}{d x})_{K}$. Thus $S^F$, as a Hermitian operator that commutes with a conjugation, has an extension to a selfadjoint operator, by an extension theorem of von Neumann (see for example proposition 13.25 in \cite{Sch}). The set of all continuous positive definite extensions is convex, hence if there are two continuous extensions, there is a one-parameter family of distinct continuous positive definite extensions. 
\end{proof}

Let $\Omega\subset\mathbb{R}^n$ be an open and connected set, and let $F\colon\Omega-\Omega\to\mathbb{C}$ be a continuous positive definite function. We denote by $\mathcal{K}_F$ the set of all finite positive Borel measures $\mu$ on $\mathbb{R}^n$ such that $\widehat{\mu}$ is an extension of the funtion $F$, i.e.,
$F(z)=\widehat{\mu}(z)$ for $z\in\Omega-\Omega$. We also denote by $\mathcal{C}_F$ the set of all systems $\{A_1,\ldots,A_n\}$ of selfadjoint strongly commuting extensions of operators $S_1^F,\ldots,S_n^F$. An immediate consequence of theorem \ref{MTh} is the following. 

\begin{cor}\label{corresp} There exists a bijective correspondence between sets $\mathcal{K}_F$ and $\mathcal{C}_F$. 
\end{cor}

\begin{thm}\label{tk1} If $\mathcal{K}_F$ contains a measure $\mu$ having its support compact in $\mathbb{R}^n$, then $\mathcal{K}_F$ is a singleton, i.e., $\mathcal{K}_F=\{\mu\}$. 
\end{thm}
\begin{proof} Let $\mu\in \mathcal{K}_F$ be such that $\textrm{supp}(\mu)\subset [-L,L]^n\subset\mathbb{R}^n$, for some $L>0$. In particular, $F$ has $\widehat{\mu}$ as an entire analytic extension to a positive definite function on $\mathbb{R}^n$. Fix $j=1,\ldots,n$. By \eqref{wip}, for any functions $\varphi\in C^\infty_0(\Omega)$, we obtain
\begin{align*}
|\langle S_j^F(F_\varphi),F_\varphi\rangle| &=\left|\int_{[-L,L]^n}\widehat{-i\frac{\partial\varphi}{\partial t_j}}(t)\widehat{\varphi}(t)d\mu(t)\right|\\
&=\left|\int_{[-L,L]^n}t_j|\widehat{\varphi}(t)|^2d\mu(t)\right|\\
&\leq L\Vert F_\varphi\Vert^2.
\end{align*}
Therefore $S_j^F\colon \mathcal{W}_F\to\mathcal{W}_F$ is a bounded operator on the dense subspace $\mathcal{W}_F\subset \mathcal{H}_F$, with $L$ being a common bound for all operators $S_k^F$, $k=1,\ldots,n$. In particular, $S_j^F\colon \mathcal{W}_F\to\mathcal{W}_F$, as a bounded and densely defined Hermitian operator, has a unique extension to a selfadjoint operator. The theorem follows by corollary \ref{corresp}.
\end{proof}

\begin{thm}\label{tk2} The following conditions are equivalent.
\begin{enumerate}
\item[(i)] $\mathcal{K}_F$ is a singleton.
\item[(ii)] $\{F_{\varphi-\Delta\varphi}\colon\varphi\in C^\infty_0(\Omega)\}$ is a dense subspace in $\mathcal{H}_F$, where $\Delta$ is the Laplace operator on $\mathbb{R}^n$: $\Delta=\sum_{j=1}^n(\partial/\partial x_j)^2$ 
\end{enumerate}
\end{thm}
\begin{proof} We recall that a Hermitian operator $T$ is essentially selfadjoint if its closure $\overline{T}$ is selfadjoint, i.e., $T^\star=T^{\star\star}$. The assertion (ii) says that the operator $\Delta_F=\sum_{j=1}^n(S_j^F)^2\colon \mathcal{W}_F\to\mathcal{W}_F$ is essentially selfadjoint. A theorem of Nelson (corollary 9.3 in \cite{Nel}) states that $\Delta_F$ is essentially selfadjoint if and only if the system of closed operators $\overline{S_1^F},\ldots,\overline{S_n^F}$ on $\mathcal{H}_F$ is a system of strongly commuting selfadjoint operators. Moreover the operator $\overline{\Delta_F}$ analytically dominates the system of operators $\{S_j^F\}_{j=1}^n$. The theorem follows.
\end{proof}

For a continuous positive definite function $F\colon\Omega-\Omega\to\mathbb{C}$, where $\Omega\subset\mathbb{R}^n$ is an open and connected set, we define 

\begin{equation}\label{defec}
\textrm{Def}(F)=\{f\in H_F\colon \Delta f=f\},
\end{equation}
where the equation $\Delta f=f$ is understood in the sense of distributions. We note that the space $\textrm{Def}(F)$ consists of analytic functions, which follows from ellipticity of $\Delta$. 
We obtain the following consequence of theorem \ref{tk2}.

\begin{cor}\label{cordef} The set $\mathcal{K}_F$ is a singleton if and only if $\textrm{Def}(F)=\{0\}$. 
\end{cor}
\begin{proof} Since $\{F_{\varphi-\Delta\varphi}\colon\varphi\in C^\infty_0(\Omega)\}$ is also a subspace of the Hilbert space $H_F$, therefore, in the view of theorem \ref{tk2}, we just need to show that if $f\in H_F$, then the following conditions are equivalent.
\begin{enumerate}
\item[(1)] For any $\varphi\in C^\infty_0(\Omega)$ we have $\langle F_{\varphi-\Delta\varphi},f\rangle_F=0$.
\item[(2)] $\Delta f=f$ in the sense of distributions, i.e., for any $\varphi\in C^\infty_0(\Omega)$ we have $\int_\Omega f(\varphi-\Delta\varphi)=0.$
\end{enumerate}
This follows from the following computation. Let $\varphi\in C^\infty_0(\Omega)$. Then, by \eqref{ka},
\begin{align*}
\langle F_{\varphi-\Delta\varphi},f\rangle_F & =\left\langle \int_\Omega F(\cdot-y)(\varphi(y)-\Delta\varphi(y))dy,f\right\rangle_F\\
&=\int_\Omega \langle F(\cdot-y)(\varphi(y)-\Delta\varphi(y)),f\rangle_F dy \\
&=\int_\Omega (\varphi(y)-\Delta\varphi(y))\langle F(\cdot-y),f\rangle_F dy \\
&=\int_\Omega (\varphi(y)-\Delta\varphi(y))f(y)dy.
\end{align*}
This finishes the proof.
\end{proof}

By corollary \ref{cordef} we see that if $\textrm{Def}(F)\neq\{0\}$, then we have two cases: (1) $F$ is extendable (see example \ref{enegx}), (2) $F$ has no extensions (see remark \ref{rudin}).

\begin{ex}\label{enegx}
We give an example of a continuous and positive definite function defined on the interval $(-1,1)\subset\mathbb{R}$ with two different extensions to globally defined continuous positive definite functions.
We define $F\colon (-1,1)\to\mathbb{R}$ by $F(x)=e^{-|x|}$, and we put
\begin{equation*}
\tilde{F}_1(x)=e^{-|x|},\quad x\in\mathbb{R};
\end{equation*} 
the covariance function for the Ornstein-Uhlenbeck process; and
\begin{equation*}
\tilde{F}_2(x)=
\begin{cases}
e^{-|x|}, & |x|\leq 1 \\
e^{-1}(2-|x|), & 1<|x|\leq 2 \\
0, & 2<|x|.
\end{cases}
\end{equation*}
Then both functions $\tilde{F}_1$ and $\tilde{F}_2$ are positive definite continuous extensions of the positive definite continuous function $F$. Moreover,
\begin{equation*}
\tilde{F}_1=\widehat{\mu_1}\quad\textrm{where}\quad d\mu_1(t)=\frac{1}{\pi}\frac{dt}{1+t^2},
\end{equation*} 
and
\begin{equation*}
\tilde{F}_2=\widehat{\mu_2}\quad\textrm{where}\quad d\mu_2(t)=A(t)dt\quad\textrm{with}\quad A(t)=\frac{1}{2\pi}\int_{-2}^2e^{itx}\tilde{F}_2(x)dx.
\end{equation*}
We see that measures $\mu_1$ and $\mu_2$ are quite different. The reason being that the function $A$ is entire analytic, and that the function $t\mapsto \frac{1}{\pi}\frac{1}{1+t^2}$ is not.

Let's look at corollary \ref{cordef} in the case when $F\colon (-1,1)\to\mathbb{C}$ is given by $F(x)=e^{-|x|}$. We view $F\colon (-1,1)\to\mathbb{C}$ as a function $F\colon\Omega-\Omega\to\mathbb{C}$ where $\Omega=(0,1)$.
By corollary \ref{cordef} we know that the subspace $\textrm{Def}(F)$ is non-zero,  as $F$ has at least two extensions. On the other hand, when $n=1$ we have $\Delta=(d/dx)^2$, and the solutions to the equation $(d/dx)^2f=f$ for continuous functions $f\colon (0,1)\to\mathbb{C}$ are of the form: $f(x)=ae^{-x}+be^x$, where $a,b\in\mathbb{C}$. The reader can check that these functions are in $H_F$, so
$$\textrm{Def}(F)=\{(0,1)\ni x\mapsto ae^{-x}+be^x\colon a,b\in\mathbb{C}\}.$$ 
\end{ex}

\begin{rem}\label{rudin} (Connection with non-existence result of Rudin (\cite{Rud1})) Rudin in \cite{Rud1} gives an example of a continuous positive definite function defined on the open square $(-1,1)\times (-1,1)\subset \mathbb{R}^2$, which is not extendible. We will relate to the non-existence example of Rudin using our language. If $F\colon\Omega-\Omega\to\mathbb{C}$ is a continuous positive definite function that admits an extension, then we obtain a system $\{A_1,\ldots,A_n\}$ of strongly commuting selfadjoint extensions of operators $S_1^F,\ldots, S_n^F$. In particular, the operator $\Delta_F=\sum_{j=1}^n(S_j^F)^2$ admits a selfadjoint extension of the form $\Delta_A=\sum_{j=1}^n A_j^2$. Combining Rudin's example with our theorem \ref{tk2}, it follows that not all selfadjoint extensions of $\Delta_F$ are of the form $\Delta_A$.
\end{rem}  

\subsection{Connection with the Fuglede conjecture (see \cite{Fug74,IKT03})}
We can extend the notion of positive definiteness from functions to distributions. The definition was introduced by L. Schwartz (\cite{LS}, see also \cite{St}). Let $\Omega\subset\mathbb{R}^n$ be an open set and let $T\in \mathcal{D}'(\Omega-\Omega)$ be a distribution. For a function $\varphi\in C^\infty_0(\Omega)$ and $x\in \Omega$ we define $\varphi_x\colon x-\Omega\to\mathbb{C}$ by $\varphi_x(x-y)=\varphi(y)$, where $y\in\Omega$. Since $x-\Omega\subset\Omega-\Omega$ we can view $\varphi_x$ as an element of the space $C^\infty_0(\Omega-\Omega)$. We put 
\begin{equation*}
T_\varphi(x)=T(\varphi_x)
\end{equation*}
for $x\in\Omega$. Then $T_\varphi$ is a smooth function on $\Omega$ (compare with the convolution of a distribution with a test function). We say that $T$ is a \textit{positive definite distribution} if for any $\varphi\in C^\infty_0(\Omega)$ we have
\begin{equation*}
\int_{\Omega} T_\varphi(x)\overline{\varphi(x)}dx\geq 0
\end{equation*} 
We consider the delta Dirac function $\delta=\delta_0$ supported at $0\in\Omega-\Omega$. For $\varphi\in C^\infty(\Omega)$ we have that $\delta_\varphi(x)=\varphi_x(0)=\varphi_x(x-x)=\varphi(x)$. Therefore the delta Dirac distribution $\delta$ is positive definite. We can carry over the construction of the Hilbert space $\mathcal{H}_F$, where $F\colon\Omega-\Omega\to\mathbb{C}$ is a continuous positive definite function, to the case of the delta Dirac function $\delta$. We get the Hilbert space $\mathcal{H}_\delta$. Since $\delta_\varphi=\varphi$, we see that $\mathcal{H}_\delta=L^2(\Omega)$. 

If $\Omega=(0,1)\subset\mathbb{R}$ is the unit interval, then the densely defined operator $-i\frac{d}{dx}\colon C^\infty_0(\Omega)\subset L^2(\Omega)\to L^2(\Omega)$ has a selfadjoint extension. Therefore we have the associated projection valued-measure $E$. Denote $\Lambda=\textrm{spec}_{\textrm{pt}}(E)=\{\lambda\in\mathbb{R}\colon E(\{\lambda\})>0\}$. It is well-known that in this case the spectrum is pure point, i.e., atomic. Let's look at the following two examples. 

(1) Let $\Omega=(0,1)\cup (2,3)\subset\mathbb{R}$. Then  $\Lambda=\{0,1/4\}+\mathbb{Z}$, and the corresponding eigenspaces have dimension one. Thus, if we denote
\begin{equation*}
E_\lambda(t)=\frac{1}{\sqrt{|\Omega|}}e^{2\pi i\lambda\,t},\quad t\in\mathbb{R},
\end{equation*}
then the set $\{E_\lambda\colon \lambda\in \Lambda\}$ is an orthonormal basis of $L^2(\Omega)$. 

We define operators $V_\Omega(t)\colon L^2(\Omega)\to L^2(\Omega)$ by $V_\Omega(t)E_\lambda=E_\lambda(t)E_\lambda$, where $t\in\mathbb{R}$. Then $\{V_\Omega(t)\}_{t\in\mathbb{R}}$ is a unitary representation of $\mathbb{R}$ on $L^2(\Omega)$ (compare with corollary \ref{intert}). Moreover, the representation $\{V_\Omega(t)\}_{t\in\mathbb{R}}$ is unitarily equivalent to the following representation $\{V_{\mu_c}(t)\}_{t\in\mathbb{R}}$ of $\mathbb{R}$ on the space $l^2(\Lambda)$, which can be also viewed as the space of $L^2$-functions of $\Lambda$ with respect to the counting measure $\mu_c$. 
For a sequence $(\xi_\lambda)_{\lambda\in\Lambda}\in l^2(\Lambda)$ define
$V_{\mu_c}((\xi_\lambda)_{\lambda\in \Lambda})=(E_\lambda(t)\xi_\lambda)_{\lambda\in\Lambda}$.
Then the above representation $\{V_{\mu_c}(t)\}_{t\in\mathbb{R}}$ is unitarily equivalent to the representation $\{V_\Omega(t)\}_{t\in\mathbb{R}}$.

(2) Let $\Omega=(0,1)\cup (3,5)\subset\mathbb{R}$. Then the situation here is quite different. We have that $\Lambda=(1/2)\mathbb{Z}$ and the corresponding eigenspaces have dimension one or two. In particular, the set $\{E_\lambda\colon \lambda\in \Lambda\}$ is not an orthonormal basis of $L^2(\Omega)$. More precisely,
$L^2(\Omega)$ is unitarily isomorphic to $l^2(\mathbb{Z})\oplus l^2((1/2)\mathbb{Z})$.

The reader is referred to the papers \cite{PW01}, and \cite{JP99} regarding orthogonal Fourier exponentials and spectral pairs
for proofs of facts in (1) and (2).
 
\subsection{Stationary Gaussian Processes}\label{SGSP}

Let $(X_s)_{s\in S}$, where $S$ is a set, be a Gaussian process with mean zero on a probability space $(\Lambda,\mathcal{F},P)$ indexed by $S$. For the definition see \cite{AJL}, \cite{IMcK}. The covariance function $K\colon S\times S\to\mathbb{R}$ of $(X_s)_{s\in S}$ is given by $K(s,t)=E(X_sX_t)$.
Suppose that $S=\mathbb{R}^n$. We will call a Gaussian process \textit{stationary} if for any $s,t,h\in\mathbb{R}^n$ the covariance  function $K$ satisfies 
\begin{equation*}
K(s+h,t+h)=K(s,t).
\end{equation*}
Then $K(s,t)=K(s-t,0)$, $s,t\in\mathbb{R}^n$. Therefore $F(s)=K(s,0)$, $s\in\mathbb{R}^n$, is a positive definite function. Since any real-valued positive definite kernel is the covariance function of a  Gaussian process (\cite{PaSc72}, see also \cite{P,PBook,PV}), which is unique up to equivalence of stochastic processes, we deduce the following.

\begin{thm}
Let $\Omega\subset \mathbb{R}^n$ be a set, let $F\colon \Omega-\Omega\to\mathbb{R}$ be a continuous positive definite function, and let $(X_s)_{s\in \Omega}$ be a Gaussian process, whose covariance function is $K_F$; $K_F(s,t)=F(s-t)$, $s,t\in \Omega$. Then $F$ has an extension to a continuous positive definite function defined on $\mathbb{R}^n$ if and only if the Gaussian process $(X_s)_{s\in \Omega}$ has an extension to a stationary Gaussian process indexed by $\mathbb{R}^n$.
\end{thm}

\subsection{Gaussian processes whose increments in mean-square are stationary}\label{sec83}

Let $\Omega\subset\mathbb{R}^n$ be a set. A function $G\colon \Omega-\Omega\to\mathbb{R}$ is called \textit{conditionally negative definite} if for any $x_1,\ldots,x_m\in\Omega$ and any $c_1,\ldots,c_m\in\mathbb{C}$, the following condition holds true:
$\sum_{j=1}^mc_j=0\quad\Longrightarrow\quad \sum_{j,k=1}^m G(x_j-x_k)c_j\overline{c_k}\leq 0$; see definition \ref{defi}.

The proof regarding extensions of positive definite functions (Theorem \ref{MTh}) carries over to the case of conditionally negative definite functions.
As we have seen in section \ref{SGSP}, stationary Gaussian stochastic processes are associated with positive definite functions. The class of Gaussian processes that are associated to conditionally negative definite functions is the following class of Gaussian processes whose increments in mean-square are stationary. A Gaussian process $\{X_s\}_{s\in \Omega}$ indexed by $\Omega$ is said to have \textit{stationary increments in mean-square} if there exists a conditionally negative definite function $G\colon\Omega-\Omega\to\mathbb{R}$ such that for any $s,t\in\Omega$
\begin{equation*}
E(|X_s-X_t|^2)=G(s-t).
\end{equation*}
For example, fractional Brownian motion is not stationary, but its mean square-increments are stationary. The Gaussian processes whose increments in mean-square are stationary have been extensively studied in \cite{AJL} and \cite{AlJo12}.

\section{Summary}

The purpose of the section below is to highlight some main interconnections between theorems in the main body of our paper. While our focus is on extendability of positive definite functions $F$, which are given only locally referring to a fixed open connected subset $\Omega$ in $\mathbb{R}^n$, it is our contention that such extensions, when they exist, involve both symmetry and harmonic analysis inherent in the starting point, the pair $(\Omega, F)$. Recall, if  $n=1$, extendability is automatic, but even then, depending on pair $(\Omega, F)$, there is a rich variety of extensions carrying an intriguing geometric structure; in a way that is quite parallel to the notion of scattering around obstacles.

Indeed, in Corollary \ref{multop} below, we show that a comparison of two distinct extensions for a fixed $(\Omega, F)$ may be cast in the language of a \textit{scattering operator} in the sense of Lax-Phillips (\cite{LaPh}).

Let $\Omega\subset\mathbb{R}^n$ be an open and connected set, and let $F\colon\Omega-\Omega\to\mathbb{C}$ be a continuous positive definite function. 
Assume that $G\colon \mathbb{R}^n\to\mathbb{C}$ is a continuous positive definite function extending $F$. Then there exists a unique finite positive Borel measure $\mu$ on $\mathbb{R}^n$ such that $\widehat{\mu}=G$. 
We recall that we have a canonical isometric embedding $\mathcal{H}_F\hookrightarrow \mathcal{H}_G$ given by $F_\varphi\mapsto G_\varphi$. Therefore $\mathcal{H}_F$ can be viewed as a closed subspace of $\mathcal{H}_G$.

Suppose that we have two extensions $G\colon\mathbb{R}^n\to\mathbb{C}$ and $H\colon\mathbb{R}^n\to\mathbb{C}$ of $F$. Let $\mu$ be a finite Borel probability measure on $\mathbb{R}^n$ such that $\widehat{\mu}=F$, and let $\nu$ be a finite Borel probability measure on $\mathbb{R}^n$ such that $\widehat{\nu}=H$. Recall that we have unitary isomorphisms $W_\mu\colon \mathcal{H}_G\to L^2(d\mu)$, and $W_\nu\colon \mathcal{H}_H\to L^2(d\nu)$. Therefore, we have the following diagram, where we denote $W^\mu=(W_\mu)|_{\mathcal{H}_F}$ and  
$W^\nu=(W_\nu)|_{\mathcal{H}_F}$.
\begin{equation*}
\xymatrix{
& \mathcal{H}_G\ar[r]^{W_\mu} & L^2(d\mu) \ar[d]^{W^\nu(W^\mu)^\star} \\
\mathcal{H}_F \ar[ru] \ar[r] & \mathcal{H}_H\ar[r]_{W_\nu} & L^2(d\nu)
} 
\end{equation*}
The following holds, which gives existence of a \textit{scattering matrix} (in the sense of \cite{LaPh}) relating any two distinct extensions for a fixed pair $(\Omega, F)$.

\begin{cor}\label{multop} There exists a function $S\colon\mathbb{R}^n\to\mathbb{C}$ such that the following two facts hold true:
\begin{enumerate}
\item $f\in L^2(d\mu)$ if and only if $S\cdot f\in L^2(d\nu)$.
\item $W^\nu(W^\mu)^\star f=S\cdot f$ for any $f\in L^2(d\mu)$.
\end{enumerate}
\end{cor}

\section*{Acknowledgements} The authors are pleased to thank the anonymous referee(s), and the editor, for careful reading of our paper, recommended improvements, for a list of corrections, and for thoughtful and very helpful suggestions.

\bibliographystyle{elsarticle-num}

\end{document}